\documentclass[11pt]{article}
\usepackage{graphicx}
\usepackage{amsmath}
\usepackage{amssymb}
\usepackage{amsfonts}
\usepackage{amsthm}
\usepackage{xcolor}
\usepackage{enumitem}
\usepackage{hyperref}
\numberwithin{equation}{section}
\def\R{\mathbb{R}}
\def\D{\mathcal{D}^{1,2}(\R^3)}

\def\P{\mathbb{P}}
\def\G{\mathcal{G}}
\def\W{\mathcal{W}}
\def\FF{\mathcal{S}_f}

\def\X{\mathbb{X}}
\def\Y{\mathbb{Y}}

\def\neweq#1{\begin{equation}\label{#1}}
\def\endeq{\end{equation}}
\def\eq#1{(\ref{#1})}
\def\eps{\varepsilon}

\def\eps{{\varepsilon}}

\hypersetup{
    colorlinks=true,
    linkcolor=blue,
    filecolor=magenta,
    urlcolor=cyan,
}

\newtheorem{theorem}{Theorem}[section]

\newtheorem{lemma}[theorem]{Lemma}
\newtheorem{proposition}[theorem]{Proposition}

\newtheorem{remark}[theorem]{Remark}
\newtheorem{example}[theorem]{Example}
\newtheorem{definition}[theorem]{Definition}

\title{Blow-up and uniqueness of Leray-Hopf solutions\\
to forced Navier-Stokes equations}
\author{Giovanni Paolo GALDI\footnote{Department of Mechanical Engineering and Materials Science -- University of Pittsburgh --
Benedum Engineering Hall 607 -- Pittsburgh, PA 15261 -- USA -- galdi@pitt.edu} -- Filippo GAZZOLA\footnote{Dipartimento di Matematica (Dipartimento di Eccellenza MUR 2023-2027)\ -- Politecnico di Milano
-- Piazza Leonardo da Vinci 32, 20133, Milano, Italy -- filippo.gazzola@polimi.it}}
\date{}
\begin{document}
\maketitle

\begin{abstract}
We prove the existence of forces and smooth initial data for which the associated Leray-Hopf solution to the three-dimensional Navier-Stokes equations is unique and global in time, satisfies the energy equality, and exhibits infinitely many (countably many) blow-up instants within a finite time interval. We consider several classes of forces and different blow-up strengths, all of which are shown to be sharp with respect to known boundedness results in the literature. Our method also yields examples of blow-up solutions for the forced three-dimensional Euler equations.

\end{abstract}

{\small
}

\baselineskip14pt

\tableofcontents

\vfill\eject

\section{Introduction}
The question of global regularity for the three-dimensional incompressible
Navier--Stokes equations is one of the central open problems in mathematical
fluid dynamics and constitutes one of the seven Millennium Prize Problems \cite{feffer}.
In the whole space $\mathbb R^3$, the forced Navier--Stokes problem reads
\begin{equation}\label{NS-intro}\begin{aligned}
&\left.
\begin{aligned}
 V_t -\Delta V+(V\cdot\nabla)V+\nabla P &= f
        \\
 \nabla\cdot V &=0\end{aligned}\right\}\ \text{in }\mathbb R^3\times(0,\infty),
        \\
&\quad V(\xi,0)=V_0(\xi),
        \ \ \xi=(x,y,z)\in\mathbb R^3.
\end{aligned}
\end{equation}
One of the formulations of the Millennium Problem \cite[Statement (C)]{feffer} asks whether there exist
smooth, divergence-free initial data $V_0$ of finite energy and a smooth forcing term $f$ for
which \eqref{NS-intro} fails to possess a global smooth solution. Despite the
enormous effort devoted to this question, the possibility of finite-time
singularity formation for smooth data and smooth forcing remains unresolved.

On the other hand, at the level of finite-energy weak solutions, global
existence has been known since the pioneering works of Leray \cite{leray} and Hopf \cite{hopf}. Recall
that, for the general forced problem \eqref{NS-intro}, a Leray--Hopf solution
is, in particular, a distributional solution satisfying\footnote{See Section
  \ref{results} for notation.}
\[
 V\in L^\infty_{\rm loc}(\mathbb R_+;L^2_\sigma(\mathbb R^3))
       \cap
       L^2_{\rm loc}(\mathbb R_+;H^1(\mathbb R^3)),
\]
together with the energy inequality
\begin{equation}\label{energy-intro}
 \frac12\|V(t)\|_{L^2(\R^3)}^2
 +\int_0^t\|\nabla V(\tau)\|_{L^2(\R^3)}^2\,d\tau
 \leq\frac12\|V(0)\|_{L^2(\R^3)}^2
 +\int_0^t\langle f(\tau),u(\tau)\rangle\,d\tau,\ \ t>0
\end{equation}
for the appropriate $f$. A precise definition, including
the assumptions on the force and the meaning of the energy inequality \eq{energy-intro}, is
given in Definition~\ref{LHsolutions} below together with the comments immediately following it. Thus, global existence in the Leray--Hopf
class does not by itself settle the fundamental questions of regularity and
uniqueness.

In this paper, we investigate these questions by allowing the forcing term to
have less than full smoothness. Precisely, we construct forcing terms $f$ and {\em smooth}
initial data for which singularities do occur in finite time. More
importantly, the regularity of the forces in our constructions is essentially
optimal: a slight improvement in their regularity, in the appropriate
function-space scale, rules out the corresponding type of singular behaviour
by known regularity results. In this sense, our examples approach from the
singular side some of the regularity thresholds separating presently known
regularity theory from the unresolved smooth-data problem.

More specifically, for an arbitrarily given  $T>0$, we construct classes of forcing terms
$f$ and smooth initial data for which the corresponding solution $(V,P)$ has
the following properties. First, $V$ is a global Leray--Hopf solution,
satisfies the energy equality -- namely, (\ref{energy-intro}) with the equality sign and all $t>0$ -- and is unique in the Leray--Hopf class.
Second, there exist an increasing sequence
\begin{equation}\label{1.3}
        \{T_m\}\subset(0,T)
\end{equation}
and a (possibly unbounded or constant) sequence of points
\begin{equation}\label{1.4}
        \{\xi_m\}\subset\mathbb R^3
\end{equation}
such that
\begin{equation}\label{1.5}
 (V,P)\ \text{is smooth in}\
 (\mathbb R^3\times\mathbb R_+)
 \setminus
 \bigcup_{m=1}^{\infty}\{(\xi_m,T_m)\}
,
\end{equation}
while $V$ develops a singularity at each $T_m$ in suitable norms, depending on the functional class where $f$ is prescribed. In other words, the solution
has  countably many blow-up instants in a finite
time interval where the strength of the singularities that can be produced
depends quantitatively on the regularity imposed on the forcing term.

A particularly significant feature of these examples is that singularity
formation occurs without any loss of uniqueness or of the energy equality.
This distinguishes the phenomenon considered here from the striking recent
developments concerning nonuniqueness of three-dimensional Navier--Stokes
solutions. Albritton, Bru\'e and Colombo \cite{albritton} constructed a forcing term for which
the same initial data give rise to distinct Leray solutions. More recently,
Coiculescu and Palasek \cite{coiculescu} established nonuniqueness for the unforced Cauchy
problem with initial data in the critical space $BMO^{-1}$, while Hou, Wang
and Yang \cite{hou} obtained nonuniqueness of Leray--Hopf solutions for the unforced
problem through a rigorous computer-assisted construction.

Our results concern a fundamentally different mechanism. The solutions
constructed here remain unique in the Leray--Hopf class and satisfy the
energy equality throughout their evolution, yet they develop singularities
at a prescribed countable collection of time--points. Consequently,
the singular behaviour exhibited in this paper cannot be attributed to
nonuniqueness of weak solutions. Rather, it is generated by the limited
regularity of the external force. Thus, {\em blow-up and uniqueness may occur simultaneously}.

Our findings are, moreover, {\em sharp} with respect to the available
regularity theory. For each class of forces considered below, we identify a
corresponding strength of blow-up and show that the regularity of the force
lies precisely at the threshold permitted by known boundedness criteria.
If the force is assumed to possess slightly better regularity, those criteria
prevent the corresponding norm of the solution from becoming unbounded.
The examples therefore complement the classical regularity theory by showing,
from the opposite direction, the essential optimality of these thresholds.
\par
The construction of our solutions is inspired by the concentration phenomena
arising in the Lane--Emden equation  \cite{lane,emden}. The basic idea can be described as
follows. We first introduce a family of solenoidal vector fields
$V=V(\xi,t)$, obtained by suitably modifying an analogous family of solutions
to the stationary Lane--Emden equation. The fields $V$ depend on the given time span $T$ and three
positive parameters, $\delta$, $\gamma$, and $\ell$, whose choice determines
their concentration and regularity properties; see \eqref{sceltaconl}. We then define the forcing
term $f$ by substituting $V$ into the left-hand side of
\eqref{NS-intro}, and using the Helmholtz-Weyl decomposition to define the pressure $P$. Thus, $V$ is, by construction, a solution of the
corresponding forced Navier--Stokes problem. By appropriately choosing the
triple $(\delta,\gamma,\ell)$, we can vary the functional class to which
$f$ belongs and, at the same time, determine the resulting regularity and
blow-up behaviour of $V$. This flexibility allows us to identify the
borderline between the regularity of the forcing term and the strength of
the singularities that the corresponding solution may exhibit.
For this argument to work, it is essential that the nonlinear term
$(V\cdot\nabla)V$ arising from this procedure is not ``too singular.'' In
fact, this is precisely what occurs, thanks to a {\em magic cancellation}
that we discuss in Remark~\ref{magic}; see also Remark~\ref{secondmagic}. This cancellation substantially
reduces the contribution of the nonlinear term and, in particular, allows
our construction to yield {\em blow-up results even for the linear Stokes
equations}; see Remark~\ref{blupmanyStokes}  below. Since the cancellation is
a consequence of the divergence-free constraint, this mechanism is specific
to the incompressible setting and, in general, is no longer available for
compressible fluid models or for certain classes of semilinear parabolic
systems.
\par
Our main results fall into five fundamental categories, according to the functional class $\mathcal C$ to which the forcing term $f$ belongs. Irrespective of the choice of $\mathcal C$, the corresponding solutions $(V,P)$ always enjoy the following properties:

\begin{itemize} \item[1.] Their initial datum \(V_0\) belongs to a suitable set of solenoidal functions \(\mathcal G_\sigma\subset C^\infty(\R^3)\); see \eqref{spaceG}. \item[2.] They are global Leray--Hopf solutions and satisfy the energy equality for every \(t>0\). \item[3.] They are unique within the class of Leray--Hopf solutions corresponding to the same data. \item[4.] They satisfy \eqref{1.3}--\eqref{1.5}. \item[5.] For \(t>T\), both the forcing term \(f\) and the corresponding solution become regular and decay uniformly to zero pointwise as \(t\to\infty\). \end{itemize}

The five categories are then distinguished by the choice of $\mathcal C$, with the corresponding solutions exhibiting blow-up in different norms.
\par
Precisely, let
$$\begin{aligned}
&\mathcal C^1_{r,q}:=\bigl\{L^1_{\rm loc}(\R_+;L^2(\R^3))\cap L^r_{\rm loc}(\R_+;L^{6/5}(\R^3))\,, r\in [2,4)\bigr\} \cup
\bigl\{L^q_{\rm loc}(\R_+;L^{2}(\R^3))\,, q\in[\tfrac54,\tfrac43)\bigr\}\,;\\
&\mathcal C^2_{r,q}:=\bigl\{L^2_{\rm loc}(\R_+;L^q(\R^3))\,,\ q\in(1,2)\bigr\}\cup\bigr\{ L^r_{\rm loc}(\R_+;L^2(\R^3))\,,\ r\in [1,2)\bigr\}\,;\\
&\mathcal C^3_{r,q}
:=
\bigcup_{(r,q)\in\mathfrak A_1}
L^r_{\rm loc}\bigl(\R_+;L^q(\R^3)\bigr)\,,\ \
\mathfrak A_1
:=
\left\{
(r,q)\in[1,\infty)\times(1,\infty):
\ q\le 3
\ \text{or}\
\frac{2}{r}+\frac{3}{q}>1
\right\}\,;\\
&\mathcal C^4_{r,q}
:=
\bigcup_{(r,q)\in\mathfrak A_2}
L^r_{\rm loc}\bigl(\R_+;L^q(\R^3)\bigr)\,,\ \ \mathfrak A_2
:=\left\{(r,q)\in[1,\infty)\times(1,\infty):q\leq\frac{3}{2}
\ \text{or}\ \frac{2}{r}+\frac{3}{q}>2\right\}\,;\\
&\mathcal C^5_{r,q}:= \bigcup_{(r,q)\in(1,\infty)^2}L^r_{\rm loc}\bigl(\R_+;L^q(\R^3)\bigr)\,.
\end{aligned}
$$
Then, for any choice of $(r,q)$ for the above classes $\mathcal C^{i}_{r,q}$, $i=1,\ldots 5$, there is $V_0\in\mathcal G_\sigma$ and   $f\in \mathcal C_{r,q}^{i}$ such that the corresponding solution $V$ satisfies the following conditions:
$$
\lim_{t\to T_m}\, \sqrt[4]{|T_m-t|}\, \|\nabla V(t)\|_{L^2(\R^3)}=
\lim_{t\to T_m}\|V(t)\|_{L^p(\R^3)}=\infty\quad\forall m\ge1\,
 \text{ and $\forall p\ge \bar{p}(r)$}\,,\ \text{in $\mathcal C^{1}_{r,q}$}\,,
$$
see Theorems \ref{main2} and \ref{main3}\,;
$$
\lim_{t\to T_m}\, {|T_m-t|^\lambda}\, \|\nabla V(t)\|_{L^2(\R^3)}=\infty\,,\ \lambda=\lambda(q,r)>0\,, \ \text{in $\mathcal C^{2}_{r,q}$}\,,
$$
see Theorem \ref{inflating}\,;
$$
\lim_{t\to T_m}\|V(t)\|_{L^\infty(\R^3)}=\infty\,,\ \text{in $\mathcal C^{3}_{r,q}$}\,,
$$
see Theorem \ref{mainLinfty}\,;
$$
\lim_{t\to T_m}\|\nabla V(t)\|_{L^\infty(\R^3)}=\infty\,,\ \text{in $\mathcal C^{4}_{r,q}$}\,,
$$
see Theorem \ref{mainLinftygrad}. Finally,
$$
\lim_{t\to T_m}\|\Delta V(t)\|_{L^\infty(\R^3)}=\infty\,,\ \text{in $\mathcal C^{5}_{r,q}$}\,;
$$
see Theorem \ref{final}. Moreover, in each of the cases considered above, we establish a corresponding blow-up result for the pressure field $P$ in suitable norms.
\par
As mentioned earlier on, the above classes are ``optimal" for the occurrence of blow-up in the indicated norms. This important property is established in Section \ref{integNS}. To this end, for
$$
f\in L^{s'}_{\rm loc}(\R_+;L^s(\R^3))\,,\ \ (s',s)\in[1,\infty)\times(1,\infty) \,,
$$
we introduce the quantity
$$
\mathcal S_f=\frac2{s'}+\frac3{s}\,.
$$
We then show that, as soon as the integrability properties of $f$, as measured by $\mathcal S_f$,
place $f$ {\em strictly} outside the corresponding class above, blow-up in the associated norm cannot occur.
In this regard, see also Remark \ref{classical}.\par
The proof of Theorem \ref{final}, which yields the blow-up of $\Delta V$,
shows that the singularities of $f$ affect only the second-order spatial
derivatives of $V$, but neither the vorticity $\nabla\wedge V$ nor the
time derivative $V_t$. This observation suggests that a blow-up result
might also hold for the {\em Euler problem}, obtained from
\eqref{NS-intro} by formally suppressing the term $\Delta V$, even for
relatively smooth forcing terms. This is indeed established in Theorem
\ref{eulertheo}, where we show that there exist
$V_0\in\mathcal G_\sigma$ and
\[
f\in C^\infty(\R^3\times[0,T))
   \cap C^{0,\alpha}(\R^3\times[0,T]),
\qquad \alpha\in(0,1),
\]
such that the corresponding Euler problem admits a solution $(V,P)$,
with
\[
V\in C^\infty(\R^3\times[0,T))
   \cap C^{0,\alpha}(\R^3\times[0,T]),
\]
that blows up at time $T$ in the sense of the Beale--Kato--Majda criterion \cite{majda}, namely,
\[
\lim_{t\to T}\int_0^t
\|\nabla\wedge V(s)\|_{L^\infty(\R^3)}\,ds=\infty.
\]

\par
We conclude this outline by noting that the problem of constructing forces leading to blow-up in different norms has also been investigated very recently by Zhang \cite{zhang}, whose results were subsequently extended and generalised by Beir\~ao da Veiga--Yang \cite{Hugo}. However, their approach and results are of a different nature from those presented here.
\par
The paper is organised as follows. In Section \ref{results}, we give a precise formulation of our main results. In Section \ref{integNS}, we employ well-known results and some of their consequences---developed in the Appendix, Section~\ref{proofsgigasohr}---to show that our blow-up findings are sharp in the sense specified previously. The backbone of the paper is presented in Section \ref{reconstruction}. There, inspired by the ``bubble'' solution of the Lane--Emden equation, we introduce a solenoidal vector field $V=V(\xi,t)$ depending on a given $T>0$ and on three parameters, $\delta$, $\gamma$, and $\ell$. We first identify the range of these parameters for which $V$ is a Leray--Hopf solution; see Lemma \ref{existencecond}. We then further restrict this range in order to guarantee that $V$ is unique in the class of Leray--Hopf solutions and satisfies the energy equality; see Lemma \ref{allterms}. Next, we characterise the range of parameters for which $V$ exhibits blow-up in different norms (Lemma \ref{wekandblup}) and then, in Lemma \ref{defforce2}, determine the summability properties of
$g:=V_t+(V\cdot\nabla)V-\Delta V$. As previously mentioned, a fundamental role in this analysis is played by the {\em magic cancellation} resulting from the solenoidality of $V$; see Lemma \ref{defforce2} and Remark \ref{magic}. The section ends with the introduction of the pressure field via Helmholtz-Weyl decomposition of the vector field $g$, and the study of its corresponding summability properties. The subsequent  Section \ref{proofs} is entirely dedicated to the proofs of our main results. They are essentially developed in two steps. On the one hand, thanks to the results established in the previous section, we prove the existence of a unique, global Leray--Hopf solution satisfying the energy equality and possessing the stated singularity properties at a single time $T>0$. On the other hand, we define a vector field $W$ as a suitably weighted series of solutions ${V_m}$, each of which possesses a singularity at a prescribed time $T_m$, and show that $W$ is a Leray--Hopf solution to \eqref{NS-intro} possessing all the properties stated in the main theorems. The final section contains an Appendix divided into two parts. In the first part, Section \ref{integLE}, we recall the stationary and evolution Lane--Emden equation and some relevant results concerning its solutions, including their uniqueness properties. In the second part, Section \ref{proofsgigasohr}, we employ the results of \cite{giga} to prove local existence theorems for \eqref{NS-intro} in maximal regularity classes, which are needed to establish the sharpness of our blow-up conditions.

\section{Main results}\label{results}

For $p\ge1$, let $L^p_\sigma(\R^3)$ be the subspace of $L^p(\R^3)$ of divergence-free vector fields. For any $T>0$, let
$$
\mathcal{D}_T=\Big\{\phi\in C^\infty_c\big(\R^3\times[0,T)\big)\mbox{ s.t. }\nabla\cdot\phi=0\mbox{ in }\R^3\times[0,T)\Big\}
$$
and we make precise what is meant by Leray-Hopf solution.

\begin{definition}\label{LHsolutions}
For given $V_0\in L^2_\sigma(\R^3)$, given $T>0$ and $f\in L^2(0,T;H^{-1}(\R^3))+L^1(0,T;L^2(\R^3))$, a vector field $V=V(\xi,t)$ is a
Leray-Hopf solution to \eqref{NS-intro} in $\R^3\times(0,T)$ if
\neweq{functionalcond}
V\in L^\infty(0,T;L^2_\sigma(\R^3))\cap L^2(0,T;H^1(\R^3))\, ,
\endeq
\neweq{weakform}
\int_{0}^{\infty}\left\{\int_{\R^3}\big[\nabla V:\nabla\phi+(V\cdot\nabla)V\cdot\phi-V\cdot\phi_t\big]\right\}\, dt=
\int_{0}^{\infty}\langle f,\phi\rangle dt+\int_{\R^3}V_0\cdot\phi(0)\qquad\forall\phi\in\mathcal{D}_T\, ,
\endeq
\neweq{energyineq}
\|V(t)\|_{L^2(\R^3)}^2+2\int_{0}^{t}\|\nabla V(s)\|_{L^2(\R^3)}^2ds\le\|V_0\|_{L^2(\R^3)}^2+2\int_{0}^{t}\langle f(s),V(s)\rangle ds
\qquad\forall t\in(0,T]\, .
\endeq
A vector field is a global Leray-Hopf solution to \eqref{NS-intro} if it is a Leray-Hopf solution for any $T>0$.
\end{definition}

Condition \eqref{energyineq} is called the energy inequality, and $\langle\cdot,\cdot\rangle$ denotes the duality between $H^{-1}(\R^3)$ and $H^1(\R^3)$. Note that, at those times $t$ for which \eqref{energyineq} holds with strict inequality, there is a loss of kinetic energy that cannot be attributed to viscous dissipation.
\par
In order to recover the pressure, the minimal condition is $P\in L^1_{\rm loc}(\R^3\times(0,T))$ and one is naturally
led to solve (in distributional sense)
\neweq{usualpressure}
-\Delta P=\nabla\cdot\big[(V\cdot\nabla)V-f\big]\qquad\mbox{in }\R^3\quad\mbox{for a.e. }t\in(0,T)\, .
\endeq

Throughout the paper, the initial velocity $V_0$ in \eq{NS-intro}$_3$ will be taken in the following space of smooth functions
\neweq{spaceG}
\G_\sigma=\{V\in\G;\ \nabla\cdot V=0\mbox{ in }\R^3\}\quad\mbox{where}\quad
\G\, =\, \bigcap_{k\in\mathbb{N}}\,\bigcap_{q\in[2,\infty]}\,W^{k,q}(\R^3)\ \subset C^\infty(\R^3)\, .
\endeq

Our first statement reads.

\begin{theorem}\label{main2}
Let $\G_\sigma$ be as in \eqref{spaceG}. For any $T>0$ and any $2\le r<4$ there exist
$$V_0\in\G_\sigma\, ,\qquad f\in L^1_{\rm loc}(\R_+;L^2(\R^3))\cap L^r_{\rm loc}(\R_+;L^{6/5}(\R^3))\subset L^2_{\rm loc}(\R_+;H^{-1}(\R^3))\, ,$$
such that $T\mapsto\|\nabla V_0\|_{L^2(\R^3)}$ is strictly decreasing and:\par\noindent
$(\rm{a})$ problem \eqref{NS-intro} admits a global Leray-Hopf solution $V\in C^0(\R_+;L^2(\R^3))$ which satisfies
the energy equality, \eqref{energyineq} with equality sign, for all $t>0$;\par\noindent
$(\rm{b})$ there exists $\overline{q}>3$ such that $V\in L^{2q/(q-3)}_{\rm loc}(\R_+;L^q(\R^3))$ for all $q>\overline{q}$ and, hence,
such solution is unique;\par\noindent
$(\rm{c})$ there exist a strictly increasing sequence $$\{T_m\}\subset (0,T)$$ and a sequence
$\{\xi_m\}\subset\R^3$ such that the solution satisfies $V,P\in C^\infty(\R^3\times\R_+\setminus\cup_{m=1}^\infty\{(\xi_m,T_m)\})$;\par\noindent
$(\rm{d})$ such solution also satisfies $V\not\in L^8_{\rm loc}(\R_+;L^4(\R^3))$, $\nabla V\not\in L^4_{\rm loc}(\R_+;L^2(\R^3))$ and
\neweq{blupofL3}
\lim_{t\to T_m}\, \sqrt[4]{|T_m-t|}\, \|\nabla V(t)\|_{L^2(\R^3)}=
\lim_{t\to T_m}\|V(t)\|_{L^p(\R^3)}=\infty\quad\forall m\ge1\, ,\quad\forall p>\frac{18}5 -\frac{12}{5r}\, ,
\endeq
while the associated pressure $P$ satisfies $P\in L^2_{\rm loc}(\R_+;L^p(\R^3))$ for all $p\ge1$,
$\nabla P\in L^r_{\rm loc}(\R_+;L^{6/5}(\R^3))$ and
$$\exists\overline{s}>r\quad\mbox{such that}\quad\lim_{t\to T_m}\|\nabla P\|_{L^s(0,t;L^{6/5}(\R^3))}=\infty
\quad\forall s>\overline{s}\quad\forall m\ge1\, .$$
\end{theorem}

We obtain the same statement for $f\in L^{r}_{\rm loc}(\R_+;L^2(\R^3))$.

\begin{theorem}\label{main3}
Let $\G_\sigma$ be as in \eqref{spaceG}. For any $T>0$ and $5/4\le r<4/3$ there exist
$$V_0\in\G_\sigma\, ,\qquad f\in L^r_{\rm loc}(\R_+;L^2(\R^3))\subset L^1_{\rm loc}(\R_+;L^2(\R^3))\, ,$$
such that $T\mapsto\|\nabla V_0\|_{L^2(\R^3)}$ is strictly decreasing and:\par\noindent
-- $(\rm{a})$-$(\rm{b})$-$(\rm{c})$ in Theorem \ref{main2} hold;\par\noindent
-- such solution also satisfies $V\not\in L^8(0,T;L^4(\R^3))$, $\nabla V\not\in L^4(0,T;L^2(\R^3))$ and
\neweq{blupofL3n2}
\lim_{t\to T_m}\, \sqrt[4]{|T_m-t|}\, \|\nabla V(t)\|_{L^2(\R^3)}=
\lim_{t\to T_m}\|V(t)\|_{L^p(\R^3)}=\infty\quad\forall m\ge1\, ,\quad\forall p>6-\frac{4}{r} \, ,
\endeq
while the associated pressure $P$ satisfies $P\in L^2_{\rm loc}(\R_+;L^p(\R^3))$ for all $p\ge1$, $\nabla P\in L^r_{\rm loc}(\R_+;L^2(\R^3))$ and
$$\exists\overline{s}>r\quad\mbox{such that}\quad\lim_{t\to T_m}\|\nabla P\|_{L^s(0,t;L^2(\R^3))}=\infty
\quad\forall s>\overline{s}\quad\forall m\ge1\, .$$
\end{theorem}

We then weaken the blow-up requests \eq{blupofL3} and \eq{blupofL3n2}, going towards the criticality condition for the external force $f$. But,
even to reach the ``weaker blow-up condition'' of the enstrophy (with no rate),
we do not cover the case where $f\in L^2_{\rm loc}(0,T;L^2(\R^3))$.

\begin{theorem}\label{inflating}
Let $\G_\sigma$ be as in \eqref{spaceG}. For any $T>0$, any $q\in(1,2)$ \big(resp.\ $r\in[1,2)$\big) there exist
\neweq{f1f2}
V_0\in\G_\sigma\, ,\qquad f\in L^2_{\rm loc}(\R_+;L^q(\R^3))\quad\Big(\mbox{resp. }f\in L^r_{\rm loc}(\R_+;L^2(\R^3))\Big)
\endeq
such that $T\mapsto\|\nabla V_0\|_{L^2(\R^3)}$ is strictly decreasing and:\par\noindent
-- $(\rm{a})$-$(\rm{b})$-$(\rm{c})$ in Theorem \ref{main2} hold;\par\noindent
-- such solution also satisfies
$$
\lim_{t\to T_m}\|\nabla V(t)\|_{L^2(\R^3)}=\infty\quad\forall m\ge1\quad\mbox{with}
$$
\neweq{quantitative}
\forall q\ge\frac{3}{2}\qquad\lim_{t\to T_m}|T_m-t|^\lambda
\|\nabla V(t)\|_{L^2(\R^3)}=\infty\quad\forall\lambda<\frac{3(2-q)}{4q}
\endeq
\neweq{quantitative2}
\bigg(\mbox{resp. }\forall r\ge\frac{4}{3}\qquad\lim_{t\to T_m}|T_m-t|^\lambda
\|\nabla V(t)\|_{L^2(\R^3)}=\infty\quad\forall\lambda<\frac{2-r}{2r}\bigg)\, ,
\endeq
while the associated pressure satisfies $P\in L^2_{\rm loc}(\R_+;L^p(\R^3))$ for all $p\ge1$,
$\nabla P\in L^2_{\rm loc}(\R_+;L^q(\R^3))$ \big(resp.\ $\nabla P\in L^r_{\rm loc}(\R_+;L^2(\R^3))$\big) with some norms of $\nabla P$
diverging to $\infty$ as $t\to T_m$.
\end{theorem}

Although the whole range $q\in(1,2)$ can be analysed in \eq{quantitative}, we consider the simpler case $q\in[\tfrac32,2)$
which, in particular, excludes $q=\tfrac65$ that plays a major role in Theorem \ref{main2}. For \eq{quantitative2}, the case
$r<\tfrac43$ is treated
in Theorem \ref{main3} and it will become clear from the proof of Theorem \ref{inflating} (see Section \ref{sharpsohr})
how important this threshold is in our approach. Both \eq{quantitative} and \eq{quantitative2} show that $(r,q)=(2,2)$ is a limit case.\par
\par
Our next results concern pointwise blow-up.
We first deal with the pointwise blow-up of $V$: in the next result we reach the threshold of the criticality of $f$,
although we do not attain it.

\begin{theorem}\label{mainLinfty}
Let $\G_\sigma$ be as in \eqref{spaceG}. For any $T>0$, any $(r,q)\in[1,\infty)\times(1,\infty)$ satisfying
\neweq{ipotesona}
\bigg[r\ge 1\mbox{ and }q>\frac32\mbox{ and }2<\frac{2}{r}+\frac{3}{q}<5\bigg]\quad\mbox{or}\quad
\bigg[r\ge1\mbox{ and either }\frac{2}{r}+\frac{3}{q}\ge5\mbox{ or }1<q\le\frac32\bigg]\, ,
\endeq
there exist $V_0\in\G_\sigma$ and $f\in L^r_{\rm loc}(\R_+;L^q(\R^3))$ such that:\par\noindent
-- $(\rm{a})$-$(\rm{b})$-$(\rm{c})$ in Theorem \ref{main2} hold;\par\noindent
-- such solution satisfies
\neweq{blupVinftym}
\lim_{t\to T_m}\|V(t)\|_{L^\infty(\R^3)}=\infty\, ,
\end{equation}
while the associated pressure $P$ satisfies $P\in L^2_{\rm loc}(\R_+;L^p(\R^3))$ for all $p\ge1$ and
$\nabla P\in L^r_{\rm loc}(\R_+;L^q(\R^3))$.
\end{theorem}

We reach criticality for $f$, and also strict subcriticality, by lowering the pointwise blow-up request to $\nabla V$.

\begin{theorem}\label{mainLinftygrad}
Let $\G_\sigma$ be as in \eqref{spaceG}. For any $T>0$, any $(r,q)\in[1,\infty)\times(1,\infty)$ satisfying
\neweq{ipotesonagrad}
\bigg[r\ge1\mbox{ and }q>3\mbox{ and }1<\frac{2}{r}+\frac{3}{q}<5\bigg]\quad\mbox{or}\quad
\bigg[r\ge1\mbox{ and either }\frac{2}{r}+\frac{3}{q}\ge5\mbox{ or }1<q\le3\bigg]\, ,
\endeq
there exist $V_0\in\G_\sigma$ and $f\in L^r_{\rm loc}(\R_+;L^q(\R^3))$ such that:\par\noindent
-- $(\rm{a})$-$(\rm{b})$-$(\rm{c})$ in Theorem \ref{main2} hold;\par\noindent
-- such solution satisfies
\neweq{blupnablaVinftym}
\lim_{t\to T_m}\|\nabla V(t)\|_{L^\infty(\R^3)}=\infty\, ,
\endeq
while the associated pressure $P$ satisfies $P\in L^2_{\rm loc}(\R_+;L^p(\R^3))$ for all $p\ge1$ and
$\nabla P\in L^r_{\rm loc}(\R_+;L^q(\R^3))$;\par\noindent
- there exist $(r,q)\in[1,\infty)\times(1,\infty)$ satisfying \eqref{ipotesonagrad}$_1$  such that
\neweq{bullet}
V\in C^0(\R^3\times\R_+)\quad\mbox{and}\quad V(\xi,T_m)\equiv0\mbox{ in }\R^3\mbox{ for all }m\ge1\, .
\endeq
\end{theorem}

We get rid of the restriction in \eq{ipotesonagrad}$_2$ by lowering further the request of pointwise blow-up to the second
spatial derivatives.

\begin{theorem}\label{final}{final}
Let $\G_\sigma$ be as in \eqref{spaceG}. For any $T>0$, any $(r,q)\in[1,\infty)\times(1,\infty)$
there exist $V_0\in\G_\sigma$ and $f\in L^r(0,T;L^q(\R^3))$ such that:\par\noindent
-- $(\rm{a})$-$(\rm{b})$ in Theorem \ref{main2} hold;\par\noindent
-- there exist a bounded strictly increasing sequence $\{T_m\}$ of positive instants $0<T_1<T_2<...<T_m<T$ and a sequence
$\{\xi_m\}\subset\R^3$ such that the solution satisfies $V,P\in C^1(\R^3\times\R_+)\cap C^\infty(\R^3\times\R_+\setminus\cup_{m=1}^\infty\{(\xi_m,T_m)\})$;\par\noindent
-- such solution satisfies $V(\xi,T_m)\equiv\nabla V(\xi,T_m)\equiv V_t(\xi,T_m)\equiv0$ in $\R^3$ for all $m\ge1$ and
\neweq{blupDeltaVinftym}
\lim_{t\to T_m}\|\Delta V(t)\|_{L^\infty(\R^3)}=\infty\, ,
\endeq
while the associated pressure $P$ satisfies $P\in L^2_{\rm loc}(\R_+;L^p(\R^3))$ for all $p\ge1$ and
$\nabla P\in L^r_{\rm loc}(\R_+;L^q(\R^3))$.
\end{theorem}

\begin{remark}\label{ttoinfty}
A careful look at their proofs allows us to complement all the above statements with
\begin{align*}
&\lim_{t\to\infty}\|f(t)\|_{L^\infty(\R^3)}=\lim_{t\to\infty}\|V(t)\|_{L^\infty(\R^3)}=
\lim_{t\to\infty}\|\nabla V(t)\|_{L^\infty(\R^3)}=
\lim_{t\to\infty}\|(V\cdot\nabla)V(t)\|_{L^\infty(\R^3)}=\\
&=\lim_{t\to\infty}\|V_t(t)\|_{L^\infty(\R^3)}=
\lim_{t\to\infty}\|\Delta V(t)\|_{L^\infty(\R^3)}=\lim_{t\to\infty}\|\nabla P(t)\|_{L^\infty(\R^3)}=0\, .
\end{align*}
Therefore, our results describe a violent thunderstorm $f$ acting on the bounded interval of time $[0,T]$ and, with possibly
different magnitudes, generating infinitely many epochs of irregularity in the fluid. Once the thunderstorm is over ($t>T$), the force and the
fluid regularise and, asymptotically and uniformly, converge to a full rest.
\end{remark}

Theorem \ref{final} and its proof show that singularities of $f$ may affect only the second space derivatives of the solution
\underline{and not} the vorticity $\nabla\wedge V$ nor  the time derivative $V_t$.
Theorem \ref{final} also suggests that if we ``eliminate'' the blow-up term $\Delta V$ and consider the Cauchy problem for the Euler equations:
\begin{equation}\label{euler}
V_t+(V\cdot\nabla)V+\nabla P=f\, ,\quad\nabla\cdot V=0\quad\mbox{ in }\R^3\times\R_+\, ,\qquad
V(\xi,0)=V_0(\xi)\quad \mbox{in }\R^3\, ,
\end{equation}
then blow-up may occur with a relatively nice force. Indeed, by exploiting further the main ideas of the present paper,
we are able to reach the blow-up criterion by Beale-Kato-Majda \cite{majda}.
Closely related to recent results in \cite{cordoba}, we prove

\begin{theorem}\label{eulertheo}
Let $\G_\sigma$ be as in \eqref{spaceG} and let $T>0$. For any $\alpha\in(0,1)$ there exist $f\in C^\infty(\R^3\times[0,T))\cap C^{0,\alpha}(\R^3\times[0,T])$ and
$V_0\in\G_\sigma\subset C^\infty(\R^3)$ such that \eqref{euler} admits a solution $(V,P)$ satisfying
$$V\in C^\infty(\R^3\times[0,T))\cap C^{0,\alpha}(\R^3\times[0,T])\, ,\quad
\lim_{t\to T}\int_{0}^{t}\|\nabla\wedge V(s)\|_{L^\infty(\R^3)}\, ds=\infty\, ,\quad P\in C^{1,\alpha}(\R^3\times[0,T])\, .$$
\end{theorem}

\section{Sharpness of the results}\label{integNS}

We recall here a number of known statements that show that the results in Section \ref{results} are sharp. To this end, we define the {\em strength of
the force} as
\neweq{defstrength}
\forall(r,q)\in[1,\infty)\times(1,\infty)\quad\forall f\in L^r_{\rm loc}(\R_+;L^q(\R^3))\, ,\qquad\FF:=\frac{2}{r}+\frac{3}{q}\, .
\endeq

We start with

\begin{proposition}\label{summary}
Let $V_0\in L^2_\sigma(\R^3)$ and let $f\in L^2_{\rm loc}(\R_+;L^{6/5}(\R^3))+L^{5/4}_{\rm loc}(\R_+;L^2(\R^3))$.
Let $V$ be a global Leray-Hopf solution to \eqref{NS-intro}, according to Definition \ref{LHsolutions}.\par
$(i)$ If $V\in L^4_{\rm loc}(\R_+;L^4(\R^3))$ then the energy equality holds, that is, \eqref{energyineq} with equality sign.\par
$(ii)$ If there exists $q>3$ such that $V\in L^{2q/(q-3)}_{\rm loc}(\R_+;L^q(\R^3))$, then $V$ is the unique Leray-Hopf solution to
\eqref{NS-intro}.
\end{proposition}
\begin{proof} For Item $(i)$, see \cite{lions} and also \cite[Theorem 4.1]{Galdi-evol}. When
$f\in L^2_{\rm loc}(\R_+;L^{6/5}(\R^3))\subset L^2_{\rm loc}(\R_+;H^{-1}(\R^3))$
or $f\in L^{5/4}_{\rm loc}(\R_+;L^2(\R^3))\subset L^1_{\rm loc}(\R_+;L^2(\R^3))$, see \cite[Theorem 1.4.1, Chapter V, p.272]{sohr}.\par
For Item $(ii)$, see Serrin \cite[Theorem 5]{serrin2} and also \cite[Theorem 1.5.1, Chapter V, p.276]{sohr}.
\end{proof}

For our purposes, the following significantly different functional setting due to Kozono-Sohr \cite[Theorem 2.6]{kozono}, is of great interest.

\begin{proposition}\label{kozsohr}
Let $V_0\in L^3_\sigma(\R^3)$, let $T_0>0$, let $f\in L^1(0,T_0;L^2(\R^3))\cap L^r(0,T_0;L^q(\R^3))$ for
some $1<r<\infty$ and $1<q\le3$ satisfying $2/r+3/q<3$. Then there exists $0<T_*\le T_0$ and a weak solution
$V$ to \eqref{NS-intro} (not necessarily being a Leray-Hopf solution, i.e., \eqref{energyineq} may not hold)
satisfying $V\in L^\infty(0,T_*;L^3(\R^n))$.
\end{proposition}

But to show sharpness of Theorems \ref{main2} and \ref{main3} we also need the following statement taken from Sohr
\cite[Theorem 4.2.2, Chapter V, p.345]{sohr}.

\begin{proposition}\label{supersohr}
Let $V_0\in\G$, let $T_0>0$, let $f\in L^{4/3}(0,T_0;L^2(\R^3))+L^4(0,T_0;H^{-1}(\R^3))$. Then there exists $0<T_*\le T_0$ and a
unique Leray-Hopf solution $V$ to \eqref{NS-intro} in $(0,T_*)$
satisfying $V\in L^8(0,T_*;L^4(\R^3))$ and $\nabla V\in L^4(0,T_*;L^2(\R^3))$.
\end{proposition}

\underline{Sharpness of Theorems \ref{main2} and \ref{main3}.} We rephrase the assumptions in Proposition \ref{kozsohr} with
$$
f\in L^1(0,T_0;L^2(\R^3))\quad\mbox{and}\quad\FF<3\, .
$$
In both Theorems \ref{main2} and \ref{main3}, we have the assumption that
$$
f\in L^1_{\rm loc}(\R_+;L^2(\R^3))\quad\mbox{and}\quad\FF>3
$$
which, combined with Proposition \ref{kozsohr}, shows that they are sharp, up to the case $\FF=3$ which remains open. Indeed,
fix any $T_0>0$ and the related $0<T_*\le T_0$ from Proposition \ref{kozsohr}. Then take $0<T<T_*$ in Theorems \ref{main2} and \ref{main3}.
By \eqref{blupofL3} and \eq{blupofL3n2} we find $V\not\in L^\infty(0,T_*;L^3(\R^3))$: in fact, there may be a countable number of blow-up
instants for the $L^3(\R^3)$-norm. Moreover, Proposition \ref{supersohr} considers cases where $\FF=3$ and the stated integrability for $V$ and
$\nabla V$ is violated in Theorems \ref{main2} and \ref{main3}. Finally, concerning the blow-up conditions for $\|V(t)\|_{L^p(\R^3)}$,
see \eq{blupofL3} and \eq{blupofL3n2}, the best one can expect is to lower $p$ until $p>2$ while for $p=3$ they are both ensured
up to the critical $r$.\qed
\par\medskip
In order to comment Theorem \ref{inflating}, we collect several contributions in literature whose full (lengthy) proof
can be found in the monograph by Sohr \cite[Theorem 1.8.1, p.296]{sohr}, where a stronger result is proved.

\begin{proposition}\label{minimalf}
Let $V_0\in H^1(\R^3)$ be divergence-free and let $f\in L^2_{\rm loc}(\R_+;L^2(\R^3))$.
Let $V$ be a global Leray-Hopf solution to \eqref{NS-intro}, according to Definition \ref{LHsolutions}.\par
If $V\in L^{2q/(q-3)}_{\rm loc}(\R_+;L^q(\R^3))$ for some $q>3$, then the solution to \eqref{NS-intro} is unique and, in particular,
$\nabla V\in L^\infty_{\rm loc}(\R_+,L^2(\R^3))$; moreover, there exists an associated pressure $P$ satisfying
$\nabla P\in L^2_{\rm loc}(\R_+;L^2(\R^3))$.
\end{proposition}

Notice that in Proposition \ref{minimalf} both $V_0$ and $f$ are more regular than needed in Definition \ref{LHsolutions}.\par\medskip
\underline{Sharpness of Theorem \ref{inflating}.} By using \eq{defstrength}, Proposition \ref{minimalf} states that for some forces
satisfying $\FF=\tfrac52$, one has $\nabla V\in L^\infty_{\rm loc}(\R_+,L^2(\R^3))$. Theorem \ref{inflating} gives examples
of $f$ satisfying $\FF>\tfrac52$ (as close to $\frac52$ as wanted) for which $\nabla V\not\in L^\infty_{\rm loc}(\R_+,L^2(\R^3))$.
Since $r=q=2$ is excluded, $\nabla P\in L^2_{\rm loc}(\R_+;L^2(\R^3))$ is not ensured.\qed
\par\medskip
For any $T>0$ put
$$\W^{r,q}_T:=\big\{U\in L_{\rm 1oc}^1(\R^3\times (0,T)):\, U\in W^{1,r}(0,T;L^q_\sigma(\R^3))\cap L^r(0,T;W^{2,q}(\R^3))\big\}.$$
As a consequence of results by Giga-Sohr \cite{giga}, in Section \ref{proofsgigasohr} we prove

\begin{theorem}\label{gigasohr0}
Let $T_0>0$ and assume that
$$
f\in L^{r}(0,T_0;L^q(\R^3))\ \mbox{with}\ \frac2r+\frac3q<2\,,\quad V_0\in L_\sigma^q(\R^3)\ \mbox{with}\ \Delta V_0\in L^q(\R^3)\,.
$$
Then there exists $T_*\in (0,T_0]$ such that \eqref{NS-intro} has one and only one solution
$V\in \mathcal W^{r,q}_{T_*}$. In particular,
\neweq{EST}
V\in L^\infty(\R^3\times(0,T_*))\,.
\endeq
\end{theorem}

\underline{Sharpness of Theorem \ref{mainLinfty}.} Theorem \ref{gigasohr0} shows that \eqref{EST} holds by assuming $\FF<2$. Theorem \ref{mainLinfty}
shows that it may happen that \eqref{EST} fails when $\FF>2$, see \eq{ipotesona}$_1$. Only the case $\FF=2$ remains open.\qed

\begin{remark}\label{sharpcaffa}
In particular, Theorem \ref{mainLinfty} states that $P\in L^{5/3}_{\rm loc}(\R_+;L^{5/3}(\R^3))$. By uniqueness, $(V,P)$ is then
the suitable solution \cite[Definition p.779]{caffa} obtained in \cite[Theorem A1]{caffa}: all the assumptions are satisfied since we can construct
$f\in L^2_{\rm loc}(\R_+;H^{-1}(\R^3))$ as in Theorem \ref{mainLinfty} and the assumption $\nabla\cdot f=0$ in \cite[Theorem A']{caffa} is not
restrictive because one may always reduce to this case thanks to the Helmholtz-Weyl decomposition in any $L^p(\R^3)$-space ($1<p<\infty$),
see \cite{fujiwara,miyakawa,simader}. Hence, Theorem \ref{mainLinfty} shows that \cite[Theorem A']{caffa} is somehow sharp: the set of singular
points $(\xi,t)$ for \eqref{NS-intro} has zero 1D-Hausdorff measure, but it may contain countable many infinite points.
Even more, Theorems \ref{mainLinftygrad} and \ref{final} show that a countable number of singularities may appear also with weaker blow-up
conditions.
\end{remark}

By weakening the assumptions and using again results by Giga-Sohr \cite{giga}, in Section \ref{proofsgigasohr} we also prove

\begin{theorem}\label{gigasohr1}
Let $T_0>0$ and assume that
$$
f\in L^{r}(0,T_0;L^q(\R^3))\ \mbox{with}\ \frac2r+\frac3q<1\,,\ \ V_0\in L_\sigma^q(\R^3)\ \mbox{with}\ \Delta V_0\in L^q(\R^3)\,.
$$
Then there exists $T_*\in (0,T_0]$ such that \eqref{NS-intro} has one and only one solution
$V\in \W^{r,q}_{T_*}$. In particular,
\neweq{ESTgrad}
\nabla V\in L^\infty(\R^3\times(0,T_*))\,.
\endeq
\end{theorem}

\underline{Sharpness of Theorems \ref{mainLinftygrad} and \ref{final}.} Theorem \ref{gigasohr1} shows that \eqref{ESTgrad} holds if $\FF<1$.
Theorem \ref{mainLinftygrad} shows that \eqref{ESTgrad} may fail when $\FF>1$, see \eq{ipotesonagrad}$_1$: therefore, it is sharp, up to the
case $\FF=1$.\par
Theorem \ref{final} states the existence of $f$ such that $\FF>0$ and $V\in C^1(\R^3\times\R_+)\setminus C^2(\R^3\times\R_+)$.
The (stationary) couple
\neweq{explicitexp}
U(\xi)=\frac{(2yz,-xz,-xy)}{(1+|\xi|^2)^{5/2}}\, ,\qquad Q(\xi)\equiv0\, ,
\endeq
satisfies $U\in C^\infty(\R^3)$ with uniformly bounded derivatives, $\nabla\cdot U=0$ in $\R^3$, and
$$
-\Delta U+(U\cdot\nabla)U+\nabla Q=\frac{35\, (2yz,-xz,-xy)}{(1+|\xi|^2)^{9/2}}-\frac{(2xy^2+2xz^2,2yz^2-yx^2,2zy^2-zx^2)}{(1+|\xi|^2)^5}=:f(\xi)
$$
so that $f\in L^\infty(\R^3)$: this gives an example where $\FF=0$ with smooth $U$.\qed

\begin{remark}\label{blupmanyStokes}
For small velocities $V$, one can linearise the Navier-Stokes equations and obtain, instead of \eqref{NS-intro}, a Cauchy problem
for the {\em linear} Stokes equations:
\begin{equation}\label{StokesCauchy}
V_t-\Delta V+\nabla P=f\, ,\quad\nabla\cdot V=0\quad\mbox{ in }\R^3\times\R_+\, ,\qquad V(\xi,0)=V_0(\xi)\quad \mbox{in }\R^3.
\end{equation}
Weak solution to \eqref{StokesCauchy} are as in Definition \ref{LHsolutions}, by omitting the nonlinearity and the energy inequality.
From Giga-Sohr \cite[Theorem 2.8]{giga}, we know that both Theorems \ref{gigasohr0} and \ref{gigasohr1} hold for \eqref{StokesCauchy}.
Since our results also hold for \eqref{StokesCauchy} (with the same proof since $(V\cdot\nabla)V$ benefits of the {\em magic cancellation}
discussed in Remark \ref{magic} below), our blow-up examples are sharp also in this linear setting.
\end{remark}

\section{Construction of the solution}\label{reconstruction}

We begin this section with a technical computation.

\begin{lemma}\label{calculus}
Let $T,\alpha,\gamma>0$, $k\in\mathbb{N}$, $p>k$. For the function $h:\R^3\times[0,T)\to\R$ defined by
$$
h(\xi,t)=\frac{(T-t)^{2\alpha}\, x^k}{\left[(T-t)^{2\gamma}+|\xi|^2\right]^{p/2}}\qquad\forall(\xi,t)\in\R^3\times[0,T)
$$
we have that
\neweq{general}
\forall q>\max\left\{1,\frac{3}{p-k}\right\}\quad\exists C_q>0\quad\mbox{s.t.}\quad\|h(t)\|_{L^q(\R^3)}^q
=C_q(T-t)^{2q\alpha+(kq+3-pq)\gamma}\quad\forall t\in[0,T)\, .
\endeq
Hence,
\neweq{iff2}
\begin{array}{rcl}
\forall r\ge1\qquad h\in L^r(0,T;L^q(\R^3)) & \Longleftrightarrow & 2q\alpha+(kq+3-pq)\gamma>-q/r\, ,\\
h\in L^\infty(0,T;L^q(\R^3)) & \Longleftrightarrow & 2q\alpha+(kq+3-pq)\gamma\ge0\, ,\\
\lim_{t\to T}\|h(t)\|_{L^q(\R^3)}=0 & \Longleftrightarrow & 2q\alpha+(kq+3-pq)\gamma>0\, .
\end{array}
\endeq
The same statements hold if $x^k$ is replaced by any homogeneous polynomial (w.r.t.\ $x,y,z$) of degree $k$.
\end{lemma}
\begin{proof} In what follows, $C>0$ denotes suitable constants that may also vary within the same equation.
Using spherical coordinates $(\rho,\theta,\phi)$, combined with the change of variables $\rho=(T-t)^{\gamma}s$ for fixed $t\in[0,T)$, we find
\begin{align*}
\|h(t)\|_{L^q(\R^3)}^q &= (T-t)^{2q\alpha}\int_{\R^3}\frac{|x|^{kq}\ d\xi}{[(T-t)^{2\gamma}+|\xi|^2]^{pq/2}}
=C(T-t)^{2q\alpha}\int_0^\infty\frac{\rho^{kq+2}\ d\rho}{[(T-t)^{2\gamma}+\rho^2]^{pq/2}} \\
&= C(T-t)^{2q\alpha+(kq+3-pq)\gamma}\int_0^\infty\frac{s^{kq+3}\ ds}{[1+s^2]^{pq/2}}=C_q(T-t)^{2q\alpha+(kq+3-pq)\gamma}
\qquad\forall q>\tfrac{3}{p-k}\, ,
\end{align*}
which proves \eq{general}. Then also the three statements in \eq{iff2} follow.\par
Finally, after switching to spherical coordinates, the computations do not change if $x^k$ is replaced by any homogeneous polynomial of degree $k$.
\end{proof}

\subsection{The velocity vector field}

When the explicit form will not be necessary, in the sequel we denote by
\neweq{Pj}
\P_j(\xi)\mbox{ any homogeneous polynomial of degree $j\ge0$ with respect to $|\xi|$.}
\endeq
For instance, $\P_1(\xi)=|\xi|$ and $\P_1(\xi)=x$ (with an abuse of notation).\par
For some $\delta,\gamma>0$ and $\ell>7/2$ to be determined, we introduce the vector field
\neweq{sceltaconl}
V(\xi,t)=\frac{(T-t)^{2\delta}}{\left[(T-t)^{2\gamma}+|\xi|^{2}\right]^{\ell/2}}\, (2yz,-xz,-xy)
\endeq
that satisfies $\nabla\cdot V=0$ in $\R^3\times[0,T)$; the lower bound $\ell>7/2$ ensures that $V(t)\in L^2(\R^3)$ for $t\in[0,T)$.
By scaling invariance, any $\ell>7/2$ in \eqref{sceltaconl} will enable us to build the blow-up examples. As we now show, a special choice
of $\ell$ slightly simplifies the computations.

\begin{lemma}\label{defforce}
{\rm {\bf (Choice of ${\mathbf{\ell}}$).}} Assume that $\ell>7/2$, let $V$ as in \eqref{sceltaconl}, let
\neweq{finalf}
g:=V_t-\Delta V+(V\cdot\nabla)V\, .
\endeq
Then for any $(q,r)\in(1,\infty)\times[1,\infty)$,
\neweq{sumg}
g\in L^r(0,T;L^q(\R^3))\ \Longleftrightarrow\ \frac{1}{r}+\frac{3\gamma}{q}>\max\big\{(\ell-2)\gamma+1-2\delta,\ell\gamma-2\delta,(2\ell-3)\gamma
-4\delta\big\}\, .
\endeq
In particular, if $\ell=5$ and $(q,r)\in(1,\infty)\times[1,\infty)$,
\neweq{sumg22}
g\in L^r(0,T;L^q(\R^3))\ \Longleftrightarrow\ \gamma<\min\big\{
\tfrac{q}{5q-3}(2\delta+\tfrac{1}{r}),\, \tfrac{q}{7q-3}(4\delta+\tfrac{1}{r}),\, \tfrac{q}{3q-3}(2\delta-\tfrac{r-1}{r})\big\}\, .
\endeq
\end{lemma}
\begin{proof} For $V$ as in \eq{sceltaconl}, we find
$$
V_t(\xi,t)=\frac{(T-t)^{2\delta-1}\, \P_2(\xi)}{\left[(T-t)^{2\gamma}+|\xi|^2\right]^{\ell/2}}+
\frac{(T-t)^{2\delta+2\gamma-1}\, \P_2(\xi)}{\left[(T-t)^{2\gamma}+|\xi|^2\right]^{1+\ell/2}}\, ,
$$
$$
\Delta V(\xi,t)=\frac{(T-t)^{2\delta+2\gamma}\, \P_{2}(\xi)+(\ell-5)(T-t)^{2\delta}\, \P_4(\xi)}{\left[(T-t)^{2\gamma}+|\xi|^2\right]^{2+\ell/2}}
\, ,\qquad (V\cdot\nabla)V(\xi,t)=\frac{(T-t)^{4\delta}\, \P_3(\xi)}{\left[(T-t)^{2\gamma}+|\xi|^2\right]^\ell}\, .
$$
We emphasise that the exponent at the denominator of $(V\cdot \nabla)V$ {\em is not} $1+\ell$ as expected when multiplying $V$ with a first
order derivative: the reason for this is a {\em magic cancellation} that occurs in computations (see Remark \ref{magic} below) and that plays
a {\em fundamental role} in our analysis. We also notice that $\Delta V$ (which does have the expected power $2+\ell/2$ at the denominator!)
has a numerator where we highlighted the coefficient $(\ell-5)$: by choosing $\ell=5>7/2$ this term vanishes and the numerator simplifies.
This choice makes $V$ in \eq{sceltaconl} similar to the Laplacian of the bubble \eq{bubble}, see \eq{anomalous} and Remark \ref{proportional}
below. But let us first conclude with the general form \eq{sumg}.\par
From \eq{iff2}$_1$ with $(\alpha,k,p,q,r)=(\delta-1/2,2,\ell,q,r)$ and
$(\alpha,k,p,q,r)=(\delta+\gamma-1/2,2,2+\ell,q,r)$, we get
\neweq{sumVt}
V_t\in L^r(0,T;L^q(\R^3))\ \Longleftrightarrow\ 2\delta-1+\left(\frac{3}{q}+2-\ell\right)\gamma>-\frac{1}{r}\, .
\endeq
From \eq{iff2}$_1$ with $(\alpha,k,p,q,r)=(\delta+\gamma,2,4+\ell,q,r)$ and $(\alpha,k,p,q,r)=(\delta,4,4+\ell,q,r)$,
we infer
\neweq{Lap}
\Delta V\in L^r(0,T;L^q(\R^3))\ \Longleftrightarrow\ 2\delta+\left(\frac{3}{q}-\ell\right)\gamma>-\frac{1}{r}\, .
\endeq
From  \eq{iff2}$_1$ with $(\alpha,k,p,q,r)=(2\delta,3,2\ell,q,r)$, we obtain
\neweq{VnablaV}
(V\cdot\nabla)V\in L^r(0,T;L^q(\R^3))\ \Longleftrightarrow\ 4\delta+\left(\frac{3}{q}+3-2\ell\right)\gamma>-\frac{1}{r}\, .
\endeq
Recalling the definition of $g$ in \eq{finalf}, from \eq{sumVt}-\eq{Lap}-\eq{VnablaV} we infer \eq{sumg} and, then, \eq{sumg22}.
\end{proof}

\begin{remark}\label{magic}
The magic cancellation that occurred while computing $(V\cdot\nabla)V$ is due to the divergence-free condition. To see this,
for some $(a,b)\in\R^2$ satisfying $ab\neq0$, consider the slightly modified vector field
$$
H(\xi,t)=\frac{(T-t)^{2\delta}}{\left[(T-t)^{2\gamma}+|\xi|^{2}\right]^{\ell/2}}\, (a\, yz,b\, xz,xy)\, ,
$$
which is divergence-free if and only if $a+b+1=0$ . The three components of the vector field $(H\cdot\nabla)H$ are
$$
\Big[(H\cdot\nabla)H\Big]^1(\xi,t)=\frac{a(T-t)^{4\delta}x}{\left[(T-t)^{2\gamma}+|\xi|^2\right]^\ell}\left(
y^2+bz^2-(a+b+1)\frac{\ell\, y^2z^2}{(T-t)^{2\gamma}+|\xi|^2}\right)\, ,
$$
$$
\Big[(H\cdot\nabla)H\Big]^2(\xi,t)=\frac{b(T-t)^{4\delta}y}{\left[(T-t)^{2\gamma}+|\xi|^2\right]^\ell}\left(
x^2+az^2-(a+b+1)\frac{\ell\, x^2z^2}{(T-t)^{2\gamma}+|\xi|^2}\right)\, ,
$$
$$
\Big[(H\cdot\nabla)H\Big]^3(\xi,t)=\frac{(T-t)^{4\delta}z}{\left[(T-t)^{2\gamma}+|\xi|^2\right]^\ell}\left(
bx^2+ay^2-(a+b+1)\frac{\ell\, x^2y^2}{(T-t)^{2\gamma}+|\xi|^2}\right)\, ,
$$
whose denominators have exponent $\ell$ (and not $\ell+1$) if and only $a+b+1=0$.
\end{remark}

The proof of Lemma \ref{defforce} suggests to take $\ell=5$ and, then, \eqref{sceltaconl} becomes
\neweq{scelta}
\ell=5\ \Longrightarrow\ V(\xi,t)=\frac{(T-t)^{2\delta}\, (2yz,-xz,-xy)}{\left[(T-t)^{2\gamma}+|\xi|^{2}\right]^{5/2}}
\endeq
so that, by  \eq{iff2}$_2$ with $(\alpha,k,p)=(\delta,2,5)$
\neweq{Vrq}
V\in L^r(0,T;L^q(\R^3))\ \Longleftrightarrow\ 2\delta+\frac{1}{r}>3\left(1-\frac{1}{q}\right)\gamma\, ,
\endeq
while by \eq{general} with $(\alpha,k,p)=(\delta,2,5)$ we find
\neweq{VLqq}
\|V(t)\|_{L^q(\R^3)}^q=C_q(T-t)^{2q\delta-3(q-1)\gamma}\qquad\forall q>1\, .
\endeq

We then put
\neweq{V0explicit}
V_0(\xi):=V(\xi,0)=\frac{T^{2\delta}\, (2yz,-xz,-xy)}{\left[T^{2\gamma}+|\xi|^2\right]^{5/2}}
\endeq
and we claim that $V_0\in\G_\sigma$, as defined in \eq{spaceG}. Indeed, the divergence-free condition is a consequence of \eq{scelta}
with $t=0$, while the integrability of $V_0$ follows from \eq{VLqq} with $t=0$ and $4\delta\ge3\gamma$. Moreover,
$$
\nabla V_0(\xi)=\frac{T^{2\delta}\, \P_1(\xi)}{\left[T^{2\gamma}+|\xi|^2\right]^{5/2}}+
\frac{T^{2\delta}\, \P_3(\xi)}{\left[T^{2\gamma}+|\xi|^2\right]^{7/2}}
$$
so that $\nabla V_0\in L^q(\R^3)$ for all $q\ge2$ by using \eq{general}. The $L^\infty$-bounds are trivial. The decay at infinity increases
for higher order derivatives, showing that $V_0\in\G_\sigma$.\par
By \eq{VLqq} and \eq{iff2}$_2$ with $(\alpha,k,p,q)=(\delta,2,5,2)$,
$$
V\in L^\infty(0,T;L^2(\R^3))\ \Longleftrightarrow\ 4\delta\ge3\gamma\, .
$$

With some care, we compute the spatial derivatives of $V$ and find
$$
v^1_x=\frac{-10xyz(T-t)^{2\delta}}{\left[(T-t)^{2\gamma}+|\xi|^2\right]^{7/2}},\quad
v^1_y=\frac{2z(T-t)^{2\delta}}{\left[(T-t)^{2\gamma}+|\xi|^2\right]^{5/2}}-
\frac{10y^2z(T-t)^{2\delta}}{\left[(T-t)^{2\gamma}+|\xi|^2\right]^{7/2}},\quad v^1_z(x,y,z,t)=v^1_y(x,z,y,t).
$$
Similarly, for the spatial partial derivatives of $v^2$ and $v^3$. Therefore,
\neweq{partialr}
|\nabla V(\xi,t)|^2=\frac{(T-t)^{4\delta+4\gamma}\, \P_2(\xi)+(T-t)^{4\delta+2\gamma}\,
\P_4(\xi)+(T-t)^{4\delta}\, \P_6(\xi)}{\left[(T-t)^{2\gamma}+|\xi|^2\right]^7}\qquad\forall t\in[0,T)\, .
\endeq
Moreover, by applying \eq{iff2}$_2$ with $(\alpha,k,p)=(\delta,3,7)$ and $(\alpha,k,p)=(\delta,1,5)$, we obtain
\neweq{partialrq}
\nabla V\in L^r(0,T;L^q(\R^3))\ \Longleftrightarrow\ 2\delta+\frac{1}{r}>\left(4-\frac{3}{q}\right)\gamma
\endeq
which yields, in particular, $\nabla V\in L^2(0,T;L^2(\R^3))\Leftrightarrow4\delta+1>5\gamma$. We have so proved

\begin{lemma}\label{existencecond}
{\rm {\bf (Conditions for existence).}} Let $\delta,\gamma>0$ satisfy $4\delta\ge3\gamma$ and $4\delta+1>5\gamma$. Then,
for any $T>0$ the vector field $V$ in \eqref{scelta} satisfies \eqref{NS-intro}$_2$ and $V\in L^\infty(0,T;L^2(\R^3))\cap L^2(0,T;H^1(\R^3))$.
\end{lemma}

This means that $V$ in \eqref{scelta} is a candidate velocity to weakly solve \eqref{NS-intro} with $V_0$ as in \eq{V0explicit}.\par
So far, $(\delta,\gamma)$ {\em are not fixed}, they will be chosen depending on the result that we need to prove.
For $V$ as in \eq{scelta}, we determine conditions on $\delta,\gamma>0$ that ensure suitable integrability properties.

\begin{lemma}\label{allterms}
{\rm {\bf (Conditions for uniqueness and energy equality).}}
Besides the assumptions in Lemma \ref{existencecond}, assume that $\delta,\gamma>0$ satisfy $4\delta+1>6\gamma$.
Then the vector field $V$ in \eqref{scelta} satisfies
\neweq{foruniq}
V\in L^4(0,T;L^4(\R^3))\qquad\mbox{and}\qquad
\exists\overline{q}>3\quad\mbox{s.t.}\quad V\in L^{2q/(q-3)}(0,T;L^q(\R^3))\quad\forall q>\overline{q}.
\endeq
\end{lemma}
\begin{proof} We first notice that the additional assumption $4\delta+1>6\gamma$ implies $4\delta+1>5\gamma$ from
Lemma \ref{existencecond} and, hence, from now on we may omit the latter.\par
From \eq{Vrq} with $(q,r)=(4,4)$ we infer that
\neweq{L4L4}
V\in L^4(0,T;L^4(\R^3))\ \Longleftrightarrow\ 8\delta+1>9\gamma\, .
\endeq
But, since $4\delta\ge3\gamma$ is assumed in Lemma \ref{existencecond}, we have that
\neweq{maledetta}
8\delta+1\ \stackrel{[4\delta\ge3\gamma]}{\ge}\ 4\delta+1+3\gamma\ \stackrel{[4\delta+1>6\gamma]}{>}\ 9\gamma\, ,
\endeq
the condition on the right of \eq{L4L4} is satisfied and, in turn, also the condition on the left; this proves \eq{foruniq}$_1$.\par
For $q>3$, from \eq{Vrq} with $r=2q/(q-3)$, we infer
\neweq{q}
V\in L^{2q/(q-3)}(0,T;L^q(\R^3))\ \Longleftrightarrow\ \big(4\delta+1-6\gamma\big)q>3-6\gamma\, .
\endeq
Since $4\delta+1>6\gamma$, this proves \eq{foruniq}$_2$.\par
It is well-known \cite[p.28]{Galdi-evol} that the condition ensuring uniqueness ($4\delta+1>6\gamma$ in our case) implies the
condition ensuring the energy equality ($8\delta+1>9\gamma$ in our case), see \eq{L4L4} and \eq{maledetta}.\end{proof}

By Proposition \ref{summary}, Lemma \ref{allterms} strengthens Lemma \ref{existencecond}: the vector $V$ in \eq{scelta} is not only a
candidate weak solution to \eq{NS-intro}, but also the {\em unique} candidate Leray-Hopf solution to \eq{NS-intro} for some
$f$ to be determined.\par
The next step is to tackle the blow-up issue in several different forms.

\begin{lemma}\label{wekandblup}
{\rm {\bf (Conditions for blow-up).}}
Under the assumptions of Lemma \ref{allterms} (and \ref{existencecond}), for $V$ as in \eqref{scelta},
we have that:
\neweq{blupL3}
\mbox{if }q>2\mbox{ then }\lim_{t\to T}\|V(t)\|_{L^q(\R^3)}=+\infty\ \Longleftrightarrow\ 2q\delta<3(q-1)\gamma\, ;
\endeq
\neweq{blupenstrophy}
\lim_{t\to T}\|\nabla V(t)\|_{L^2(\R^3)}=+\infty\ \Longleftrightarrow\ 4\delta<5\gamma\, ;
\endeq
\neweq{blupgrad4}
\liminf_{t\to T}\, \sqrt[4]{T-t}\, \|\nabla V(t)\|_{L^2(\R^3)}>0\ \Longleftrightarrow\ 8\delta+1\le10\gamma\, ;
\endeq
\neweq{lastbutnotleast}
\mbox{if }\delta\le\frac12\mbox{ then }V\not\in L^8(0,T;L^4(\R^3))\mbox{ and }\nabla V\not\in L^4(0,T;L^2(\R^3))
\ \Longleftrightarrow\ 8\delta+1\le10\gamma\, .
\endeq
\end{lemma}
\begin{proof} From \eq{VLqq} we deduce \eq{blupL3}. As a consequence of Lemma \ref{calculus} and \eq{partialr}, we find
\neweq{nablaV2}
\|\nabla V(t)\|_{L^2(\R^3)}^2=C(T-t)^{4\delta-5\gamma}\qquad\forall t\in[0,T)\, ;
\endeq
hence, we obtain \eq{blupenstrophy}.
Moreover, we infer \eq{blupgrad4} from \eq{nablaV2}. Finally, by \eq{Vrq} and \eq{partialrq},
$$
V\not\in L^8(0,T;L^4(\R^3))\mbox{ and }\nabla V\not\in L^4(0,T;L^2(\R^3))
\ \Longleftrightarrow\ \gamma\ge\max\left\{\frac{16\delta+1}{18},\frac{8\delta+1}{10}\right\}=\frac{8\delta+1}{10}\, ,
$$
the last equality being true whenever $\delta\le1/2$: this proves \eq{lastbutnotleast}.
\end{proof}

When $f=0$, it is known that the three conditions \eq{blupL3} (with $q=3$), \eq{blupenstrophy}, \eq{blupgrad4} are equivalent,
see \cite[p.246, p.227]{leray}, \cite[Theorems 6.4 and 7.3]{Galdi-evol}, \cite{escauriaza,seregin}.
In general, both \eq{blupL3} (with $q\le6$) and \eq{blupgrad4} imply \eq{blupenstrophy}.
Condition \eq{blupL3} is related to Proposition \ref{kozsohr} and will be analysed in Sections \ref{seiquinti} and \ref{prooftheomain3}.
Condition \eqref{blupenstrophy} will be analysed in Section \ref{sharpsohr}.
Quite surprisingly, \eqref{blupgrad4} is equivalent to \eq{lastbutnotleast}, at least when $\delta\le1/2$, a bound that will be
satisfied in the proofs of Theorems \ref{main2} and \ref{main3}.\par
Assuming the blow-up condition \eqref{blupL3} when $q=3$, Lemmas \ref{existencecond}-\ref{allterms}-\ref{wekandblup}
all hold whenever
\neweq{capital}
0<\delta<\gamma<\min\left\{\frac{4\delta}{3},\frac{4\delta+1}{6}\right\}\, .
\endeq
This defines a nonempty open triangle that will change during the course, both enlarging it or shrinking it, depending on the result that
we need to prove. The first step is to shrink it, after particularising Lemma \ref{defforce} with the following statement.

\begin{lemma}\label{defforce2}
{\rm {\bf (Integrability of ${\mathbf{g}}$).}} Assume \eqref{capital}, let $V$ be as in \eqref{scelta},
let $g$ be as in \eqref{finalf}. Then,
\neweq{sumg1}
g\in L^2(0,T;L^{6/5}(\R^3))\ \Longleftrightarrow\ \frac13<\delta<\gamma<\min\left\{\frac{4\delta+1}{6},4\delta-1,\frac12\right\} \, ,
\endeq
\neweq{sumg2}
g\in L^{5/4}(0,T;L^2(\R^3))\ \Longleftrightarrow\ \frac25<\delta<\gamma<
\min\left\{\frac{4\delta+1}{6},\frac43 \delta-\frac{2}{15},\frac{1}{2}\right\}\, .
\endeq
\end{lemma}
\begin{proof} For $V$ as in \eq{scelta} and arguing as in the proof of Lemma \ref{defforce}, we obtain
\neweq{Vt}
V_t(\xi,t)=\frac{(T-t)^{2\delta-1}\, \P_2(\xi)}{\left[(T-t)^{2\gamma}+|\xi|^2\right]^{5/2}}+
\frac{(T-t)^{2\delta+2\gamma-1}\, \P_2(\xi)}{\left[(T-t)^{2\gamma}+|\xi|^2\right]^{7/2}}\, ,
\endeq
\neweq{summable}
\Delta V(\xi,t)=\frac{(T-t)^{2\delta+2\gamma}\, \P_{2}(\xi)}{\left[(T-t)^{2\gamma}+|\xi|^2\right]^{9/2}}
\, ,\qquad (V\cdot\nabla)V(\xi,t)=\frac{(T-t)^{4\delta}\, \P_3(\xi)}{\left[(T-t)^{2\gamma}+|\xi|^2\right]^5}\, .
\endeq

From \eq{Vt} we infer
\neweq{blup4}
4\delta>\gamma+1\ \Longleftrightarrow\ V_t\in L^2(0,T;L^{6/5}(\R^3))\, .
\endeq
From \eq{summable} we infer
\neweq{blup2}
4\delta+1>6\gamma\ \Longrightarrow\ 4\delta+1>5\gamma\ \Longleftrightarrow\ \Delta V\in L^2(0,T;L^{6/5}(\R^3))\, .
\endeq
The constraint $4\delta+1>5\gamma$ already appears in Lemma \ref{existencecond}, as expected because
similar arguments apply for the linear Stokes parabolic equation. From \eq{VnablaV} and \eq{summable} we infer
\neweq{blup3}
8\delta+1>9\gamma\ \Longleftrightarrow\ (V\cdot\nabla)V\in L^2(0,T;L^{6/5}(\R^3))\, ,
\endeq
and the inequality on the left is satisfied in view of \eq{maledetta}: it already appears in \eq{L4L4}, which
had to be expected since the proof of Proposition \ref{summary} $(ii)$ is based on the fact that one can take $\phi=V$ in \eq{weakform}.\par
If the inequality in \eq{blup4} holds, we may complement \eq{capital} with bounds for $\gamma$ and $\delta$:
\neweq{lbgamma2}
\left.
\begin{array}{rcl}
4\gamma+1\ \stackrel{\eqref{capital}}{>}\ 4\delta+1\ \stackrel{\eqref{capital}}{>}\ 6\gamma & \Longrightarrow & \delta<\gamma<\frac12\\
4\delta\ \stackrel{\eqref{blup4}}{>}\ \gamma+1\ \stackrel{\eqref{capital}}{>}\ \delta+1 & \Longrightarrow & \frac13<\delta<\gamma
\end{array}\right\}\ \Longrightarrow\ \frac13<\delta<\gamma<\frac12\, .
\endeq
Putting together \eq{blup4}-\eq{blup2}-\eq{blup3}-\eq{lbgamma2}, we infer \eq{sumg1}.\par
From \eq{Vt} we infer
\neweq{blup7}
V_t\in L^{5/4}(0,T;L^2(\R^3))\ \Longleftrightarrow\ \gamma<\frac43 \delta-\frac{2}{15}\, .
\endeq
From and \eq{summable} we infer
\neweq{blup5}
\Delta V\in L^{5/4}(0,T;L^2(\R^3))\ \Longleftrightarrow\ \gamma<\frac47 \delta+\frac{8}{35},\quad
(V\cdot\nabla)V\in L^{5/4}(0,T;L^2(\R^3))\ \Longleftrightarrow\ \gamma<\frac{8}{11}\delta+\frac{8}{55}.
\endeq

Then we impose \eq{capital} and we notice that $6\delta<6\gamma<4\delta+1$ implies first $\delta<1/2$ and then $\gamma<1/2$;
in such range,
$$
\min\left\{\frac43 \delta-\frac{2}{15},\frac47 \delta+\frac{8}{35},\frac{8}{11}\delta+\frac{8}{55}\right\}
=\min\left\{\frac43 \delta-\frac{2}{15},\frac{8}{11}\delta+\frac{8}{55}\right\}\, .
$$
Altogether the above conditions yield
$$
\frac25 <\delta<\gamma<\min\left\{\frac43 \delta,\frac{4\delta+1}{6},\frac43 \delta-\frac{2}{15},\frac{8}{11}\delta+\frac{8}{55},\frac{1}{2}\right\}
=\min\left\{\frac{4\delta+1}{6},\frac43 \delta-\frac{2}{15},\frac{1}{2}\right\}\, .
$$
Therefore, \eq{blup5} play no role and only the constraint in \eq{blup7} remains: this proves \eq{sumg2}.\end{proof}

\subsection{The survived triangles and the pressure}\label{press}

The necessary and sufficient condition for both \eq{capital} and \eq{sumg1} to hold is
\neweq{cns}
(\delta,\gamma)\in A:=\left\{(\delta,\gamma)\in\R^2_+;\ \frac13<\delta<\gamma
<\min\left[\frac{4\delta+1}{6},4\delta-1,\frac12 \right]\right\}
\endeq
which is the (nonempty!) open triangle depicted in Figure \ref{finalmente} (left).\par
Similarly, the necessary and sufficient condition for both \eq{capital} and \eq{sumg2} to hold is
\neweq{cns2}
(\delta,\gamma)\in B:=\left\{(\delta,\gamma)\in\R^2_+;\ \frac25<\delta<\gamma<\min\left[
\frac{4\delta+1}{6},\frac{4\delta}{3}-\frac{2}{15},\frac12\right]\right\}
\endeq
which is the (nonempty!) open triangle depicted in Figure \ref{finalmente} (right).

\begin{figure}[ht]
\begin{center}
\includegraphics[height=50mm]{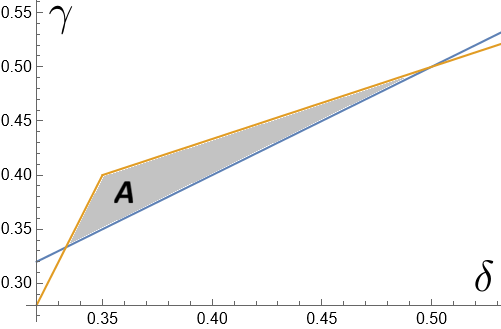}\qquad\qquad\includegraphics[height=50mm]{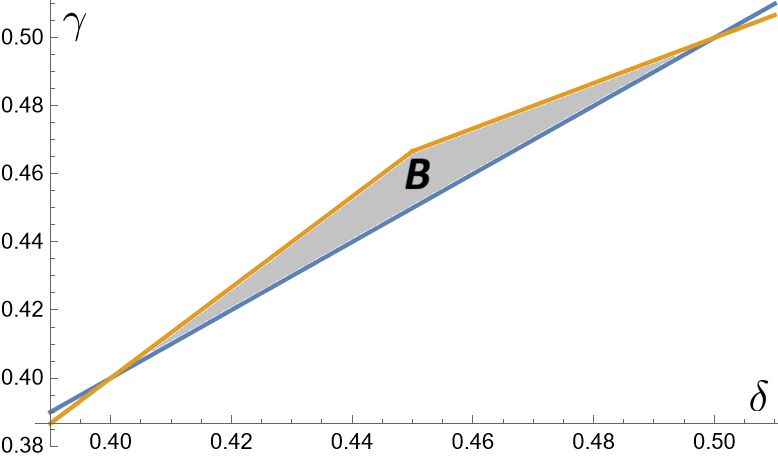}
\end{center}\vskip-5mm
\caption{The survived triangles $A$ and $B$ for $(\delta,\gamma)$ satisfying \eq{sumg1} or \eq{sumg2}.}\label{finalmente}
\end{figure}

\begin{remark}\label{secondmagic}
Without the magic cancellation, the power of the denominator of $(V\cdot\nabla)V$ in \eqref{summable} would be $6$, that is, $p=12$ instead of
$p=10$. In turn, the conditions in \eqref{blup3} and \eqref{blup5} would become, respectively, $\gamma<\tfrac{8\delta+1}{13}$ and
$\gamma<\tfrac{40\delta+8}{75}$; the triangles $A$ and $B$ would then be empty.
\end{remark}

We now introduce the pressure $P$. Since we do not yet have an explicit $f$, we cannot solve \eq{usualpressure}. Taking advantage of the facts that $V_t$ and $\Delta V$
are solenoidal in $\R^3\times\R_+$, we solve instead $-\Delta P=\nabla\cdot\big((V\cdot\nabla)V\big)$ in distributional form
and then put $f:=g+\nabla P$. To this end, we first make precise \eq{summable}$_2$ and obtain
$$
(V\cdot\nabla)V(\xi,t)=\frac{(T-t)^{4\delta}}{\left[(T-t)^{2\gamma}+|\xi|^2\right]^{5}}\, \Big(\!\!-2x(y^2+z^2),y(x^2-2z^2),z(x^2-2y^2)\Big)\, ;
$$
then we compute
$$
\nabla\cdot\big[(V\cdot\nabla)V\big] =2(T-t)^{4\delta}\, \frac{x^4+4x^2(y^2+z^2)-2y^4+16y^2z^2-2z^4}{[(T-t)^{2\gamma}+|\xi|^2]^6}+
2(T-t)^{4\delta+2\gamma}\, \frac{x^2-2y^2-2z^2}{[(T-t)^{2\gamma}+|\xi|^2]^6}\, .
$$
After convolution with the fundamental solution of $-\Delta$, we obtain
\begin{align}
P(\xi,t)= & \frac{2(T-t)^{4\delta}}{4\pi}\int_{\R^3}\frac{\zeta_1^4+4\zeta_1^2(\zeta_2^2+\zeta_3^2)
-2\zeta_2^4+16\zeta_2^2\zeta_3^2-2\zeta_3^4}{[(T-t)^{2\gamma}+|\zeta|^2]^6}\, \frac{d\zeta}{|\zeta-\xi|} \notag \\
&+\frac{2(T-t)^{4\delta+2\gamma}}{4\pi}\int_{\R^3}\frac{\zeta_1^2-2\zeta_2^2-2\zeta_3^2}{[(T-t)^{2\gamma}+|\zeta|^2]^6}
\, \frac{d\zeta}{|\zeta-\xi|}\, . \label{pressure}
\end{align}
We emphasise that the integrand is smooth in space, with a removable singularity at $\zeta=\xi$: in fact,
$P\in C^\infty(\R^3\times[0,T))$.
From \eq{Vrq} and $4\delta+1>6\gamma$ we infer
$V\in L^2(0,T;L^{2p}(\R^3))$ for $p\in[1,\infty)$ and, hence,
\neweq{regP53}
P\in L^2(0,T;L^p(\R^3))\qquad\forall p\in[1,\infty)\, .
\endeq

We then differentiate \eq{pressure} under the integral sign and obtain the vector field
\begin{align}
\nabla P(\xi,t)= & \frac{(T-t)^{4\delta}}{2\pi}\int_{\R^3}\frac{\zeta_1^4+4\zeta_1^2(\zeta_2^2+\zeta_3^2)
-2\zeta_2^4+16\zeta_2^2\zeta_3^2-2\zeta_3^4}{[(T-t)^{2\gamma}+|\zeta|^2]^6}\, \frac{\zeta-\xi}{|\zeta-\xi|^3}\, d\zeta \notag\\
&+\frac{(T-t)^{4\delta+2\gamma}}{2\pi}\int_{\R^3}\frac{\zeta_1^2-2\zeta_2^2-2\zeta_3^2}{[(T-t)^{2\gamma}+|\zeta|^2]^6}
\frac{\zeta-\xi}{|\zeta-\xi|^3}\, d\zeta \label{nablaP}
\end{align}
and, again, the integrands are smooth in space so that $\nabla P\in C^\infty(\R^3\times[0,T))$.
Using once more $4\delta+1>6\gamma$, by \eq{iff2}$_1$ with $(\alpha,k,p,q,r)\in\{(2\delta,4,12,3/2,1);(2\delta+\gamma,2,12,3/2,1)\}$ we infer
$\Delta P\in L^1(0,T;L^{3/2}(\R^3))$ so that, by \cite[Theorem 4.1]{Galdimeta}, $P\in L^1(0,T;\mathcal{D}^{2,3/2}(\R^3))$. Hence,
$\nabla P\in L^1(0,T;\mathcal{D}^{1,3/2}(\R^3))\subset L^1(0,T;L^3(\R^3))$ by Sobolev embedding; by \eq{nablaP} we see that $\nabla P$ has
``enough decay'' as $|\xi|\to\infty$, possible singularities for $t=T$ may arise only at the origin $\xi=0$ and, hence,
$\nabla P\in L^1(0,T;\mathcal{D}^{1,3/2}(\R^3))\subset L^1(0,T;L^2(\R^3))$. Summarising, by the Helmholtz-Weyl decomposition
\cite{fujiwara,miyakawa,simader}, we obtain
\begin{align}
&(V\cdot\nabla)V=-\nabla P+\Big[\underbrace{(V\cdot\nabla)V+\nabla P}_{F}\Big]\quad\mbox{and}\quad\nabla\cdot F=0\quad\mbox{with} \notag\\
&\nabla P\in L^2(0,T;L^{6/5}(\R^3))\cap L^1(0,T;L^2(\R^3))\, ,\quad F\in L^2(0,T;L^{6/5}(\R^3))\label{regnablaP}
\end{align}

In the same way we can argue in the $L^{5/4}(0,T;L^2(\R^3))$-setting, which means that the triangles $A$ and $B$
in \eq{cns}-\eq{cns2} are not modified after introduction of the pressure $P$.

\section{Proofs}\label{proofs}

\subsection{Proof of Theorem \ref{main2}}\label{seiquinti}

We use the computations in Section \ref{reconstruction} and, in order to explain in detail how we use them,
we first prove Theorem \ref{main2} for solutions to \eq{NS-intro} that blow up in a {\em unique} instant (and point).
Moreover, we need the integrability properties on $f$ to satisfy both the {\em first minimal requests} in
Definition \ref{LHsolutions} and Proposition \ref{summary}. This is why we take
$f\in L^2_{\rm loc}(\R_+;L^{6/5}(\R^3))\subset L^2_{\rm loc}(\R_+;H^{-1}(\R^3))$ as a starting assumption:
the time-exponent will then be increased with no harm and this is be done at the end of the present subsection.

\begin{theorem}\label{main}
Let $\G_\sigma$ be as in \eqref{spaceG}, let $A$ be as in \eqref{cns}. For any $(\delta,\gamma)\in A$ and $T>0$ there exist
$$ V_0\in\G_\sigma\, ,\qquad f\in L^1_{\rm loc}(\R_+;L^2(\R^3))\cap L^2_{\rm loc}(\R_+;L^{6/5}(\R^3))\subset L^2_{\rm loc}(\R_+;H^{-1}(\R^3))\, ,$$
such that $T\mapsto\|\nabla V_0\|_{L^2(\R^3)}$ is decreasing and \eqref{NS-intro}:\par\noindent
-- admits a unique global Leray-Hopf solution $V\in C^0(\R_+;L^2(\R^3))\cap L^2_{\rm loc}(\R_+;H^1(\R^3))$, which satisfies
the energy equality, \eqref{energyineq} with equality sign, for all $t>0$;\par\noindent
-- such solution satisfies
$$\lim_{t\to T}\|V(t)\|_{L^3(\R^3)}=\lim_{t\to T}\|\nabla V(t)\|_{L^2(\R^3)}=\infty\, ;$$
-- the associated pressure satisfies $P\in L^2_{\rm loc}(\R_+;L^p(\R^3))$ for all $p\in[1,\infty)$,
$\nabla P\in L^2_{\rm loc}(\R_+;L^{6/5}(\R^3))$ and
$$\exists\overline{r}=\overline{r}(\delta,\gamma)>2\quad\mbox{such that}\quad\lim_{t\to T}\|\nabla P\|_{L^r(0,t;L^{6/5}(\R^3))}=\infty
\quad\forall r>\overline{r}\, ;$$
-- such solution satisfies $V,P\in C^\infty(\R^3\times\R_+\setminus\{(0,T)\})$, that is, it admits a unique blow-up instant.
\end{theorem}
\begin{proof} Let $T>0$ and $(\delta,\gamma)\in A$, see \eq{cns}. Let $V$ and $V_0$ be as in \eq{scelta}, let $P$ be as in \eq{pressure}
implying \eq{regP53}. The fact that $V_0\in\G_\sigma$ is proved just after its definition in \eq{V0explicit}.
Using \eqref{general} for $t=0$ with $(\alpha,k,p,q)=(\delta,3,7,2)$ we obtain $\|\nabla V_0\|_{L^2(\R^3)}^2=C T^{4\delta-5\gamma}$,
which is decreasing since $(\delta,\gamma)\in A$.
Let $g$ be as in \eq{finalf}, let $f:=g+\nabla P$: then, by Lemma \ref{defforce2} and \eq{regnablaP},
$f\in L^2(0,T;L^{6/5}(\R^3))\subset L^2(0,T;H^{-1}(\R^3))$. Moreover, $f\in L^1(0,T;L^2(\R^3))$
follows from \eq{regnablaP} and from \eq{sumg22} with $(r,q)=(1,2)$, that is,
$$
g\in L^1(0,T;L^2(\R^3))\ \Longleftrightarrow\ \gamma<\min\big\{
\tfrac{2}{7}(2\delta+1),\tfrac{2}{11}(4\delta+1),\tfrac{4}{3}\delta\big\}
$$
and the inequality is satisfied since $(\delta,\gamma)\in A$.\par
By Lemma \ref{existencecond} we know that $V\in L^\infty(0,T;L^2(\R^3))\cap L^2(0,T;H^1(\R^3))$ while by Lemma \ref{allterms},
combined with Proposition \ref{summary} $(i)$, $V$ satisfies the energy equality: hence, $V$ is a Leray-Hopf solution to
\neweq{explicit}
\left\{\begin{array}{ll}
V_t-\Delta V+(V\cdot\nabla)V+\nabla P=f\, ,\quad\nabla\cdot V=0\quad & \mbox{ in }\R^3\times[0,T)\\
V(\xi,0)=V_0(\xi)\quad & \mbox{ in }\R^3\, .
\end{array}\right.
\endeq

By \eq{foruniq}$_2$, the solution to \eq{explicit} is unique by Proposition \ref{summary}
$(ii)$. The blow-up of the $L^3(\R^3)$-norm of $V(t)$ (condition \eq{blupL3}) follows from Lemma \ref{wekandblup} and \eq{cns}; this also implies
the blow-up of the $L^2(\R^3)$-norm of $\nabla V(t)$ (condition \eq{blupenstrophy}).
The blow-up of the $L^r(0,t;L^{6/5}(\R^3))$-norm of $\nabla P$ follows from \eq{nablaP}.\par
Since $V(\xi,T)=0$ in $\R^3\setminus\{0\}$, if we extend $V$ in \eq{scelta} also for $t\ge T$, by \eq{iff2}$_3$ we see that
\neweq{continuity}
V\in C^0_w(\R_+;L^2(\R^3))\ \Longleftrightarrow\ 4\delta\ge3\gamma\, ,\qquad V\in C^0(\R_+;L^2(\R^3))\ \Longleftrightarrow\ 4\delta>3\gamma\, .
\endeq
Then $V$ in \eq{scelta}, with $(T-t)$ replaced by $|T-t|$, may be extended for $t\ge T$ as a vector field
$V\in C^0(\R_+;L^2(\R^3))\cap L^2_{\rm loc}(\R_+;H^1(\R^3))$ and its extension is a global Leray-Hopf solution to
\neweq{explicit2}
\left\{\begin{array}{ll}
V_t-\Delta V+(V\cdot\nabla)V+\nabla P=f\, ,\quad\nabla\cdot V=0\quad & \mbox{ in }\R^3\times(T,+\infty)\\
V(\xi,T)=0\quad & \mbox{ in }\R^3\, ,
\end{array}\right.
\endeq
according to Definition \ref{LHsolutions}; in \eq{explicit2} also $\nabla P\in L^2_{\rm loc}(\R_+;L^{6/5}(\R^3))$, after extension of
\eq{pressure} and \eq{nablaP} for $t>T$.
The condition \eq{L4L4} ensures the validity of the energy equality
\neweq{energyeq}
\|V(t)\|_{L^2(\R^3)}^2+2\int_{T}^{t}\|\nabla V(s)\|_{L^2(\R^3)}^2ds=2\int_{T}^{t}\int_{\R^3}V(s)f(s)ds
\qquad\forall t\ge T
\endeq
also {\em beyond} $t=T$, see Proposition \ref{summary} $(i)$. By gluing together $V$ and its extension (with $V(\xi,T)=0$ in
$\R^3\setminus\{0\}$) we obtain a global Leray-Hopf solution $V\in L^\infty_{\rm loc}(\R_+;L^2(\R^3))\cap L^2_{\rm loc}(\R_+;H^1(\R^3))$,
whose uniqueness is (again) ensured by Proposition \ref{summary} $(ii)$ combined with \eq{foruniq}$_2$.\par
The smoothness of the extended $(V,P)$ follows from \eq{scelta} and \eq{pressure} with $(T-t)$ replaced by $|T-t|$.
\end{proof}

\begin{example}\label{exple}
{\em For $V$ as in \eqref{scelta}, take
$$
\delta=\frac{9}{25}=0.36\, ,\quad\gamma=\frac{2}{5}=0.4\, ,\quad\left(\frac{9}{25},\frac{2}{5}\right)\in A\, ,
$$
so that \eqref{q} yields $q>15$. Choosing $q=18$ gives $2q/(q-2)=12/5$ and
$V\in L^{12/5}_{\rm loc}(\R_+;L^{18}(\R^3))$ for
$$
V(\xi,t)=\left(\frac{2yz|T-t|^{18/25}}{\left[|T-t|^{4/5}+|\xi|^2\right]^{5/2}},
\frac{-xz|T-t|^{18/25}}{\left[|T-t|^{4/5}+|\xi|^2\right]^{5/2}},
\frac{-xy|T-t|^{18/25}}{\left[|T-t|^{4/5}+|\xi|^2\right]^{5/2}}\right)\, .
$$
Let $P(\xi,t)$ be as in \eqref{pressure}-\eqref{nablaP}, let $f=g+\nabla P$.
Then, Theorem \ref{main} and its proof show that $(V,P)$ is the unique global Leray-Hopf solution to \eqref{NS-intro} and
$\|V(t)\|_{L^3(\R^3)}$ blows up as $t\to T$.}
\end{example}

\begin{remark}
Maintaining $T>0$ fixed, Theorem \ref{main} generates a bijective relation between the triangle $A$ and the couples
$(V_0,f)\in\G_\sigma\times L^2_{\rm loc}(\R_+;L^{6/5}(\R^3))$, thereby giving a continuum of examples of Leray-Hopf solutions with blow-up.
This somehow complements the existence of strong solutions for all $f$ in a dense set of a suitable functional space proved
by Fursikov \cite{fursikov,fursikov2}, that is, a generical property.
\end{remark}

To continue, we consider the shrinked triangle
$$
A_2:=\left\{(\delta,\gamma)\in A;\, \frac{8\delta+1}{10}<\gamma\right\}
$$
which is nonempty and has vertices at $(\delta,\gamma)\in\{(11/32,3/8);(7/20,2/5);(1/2,1/2)\}$: notice the {\em strict} lower bound.
Since $(\delta,\gamma)\in A_2$ implies $\delta<1/2$, taking $(\delta,\gamma)\in A_2$ yields the blow-up conditions in both \eq{blupgrad4}
and \eq{lastbutnotleast}, that is, the solution determined in Theorem \ref{main} also satisfies
$$
\lim_{t\to T}\, \sqrt[4]{T-t}\, \|\nabla V(t)\|_{L^2(\R^3)}=\infty\, ,\quad V\not\in L^8(0,T;L^4(\R^3))\, ,\quad
\nabla V\not\in L^4(0,T;L^2(\R^3))
$$
that are three of the four blow-up conditions in Theorem \ref{main2} when a unique blow-up instant $T>0$ exists.\par
What remains to be done is to let $r\in[2,4)$ vary, to prove the last blow-up condition \eq{blupofL3}$_2$, and
to exhibit countably many blow-up instants.

\begin{proof}[Proof of Theorem \ref{main2}]

In order to increase $r$, we first rewrite \eq{sumg22} in the particular case $q=6/5$:
$$
g\in L^r(0,T;L^{6/5}(\R^3))\ \Longleftrightarrow\ \gamma<\min\big\{
\tfrac{2}{5}(2\delta+\tfrac{1}{r}),\tfrac{2}{9}(4\delta+\tfrac{1}{r}),2(2\delta-\tfrac{r-1}{r})\big\}\, .
$$
Taking $r\ge2$, the triangle $A_2$ should then be further shrink and replaced by the (more stringent) conditions
\neweq{startup}
\frac{3}{8}<\frac{8\delta+1}{10}<\gamma<\min\left\{\frac{4\delta}{3},\frac{4\delta+1}{6},\frac{2}{5}\left(2\delta+\frac{1}{r}\right),
\frac{2}{9}\left(4\delta+\frac{1}{r}\right),2\left(2\delta-1+\frac{1}{r}\right),\frac12 \right\}\, .
\endeq
Since $\delta>11/32$, we have $8\delta>4\delta+1$ and the first term inside the min can be omitted. Also notice that
$$
\min\left\{\tfrac{4\delta+1}{6},2\left(2\delta-1+\tfrac{1}{r}\right)\right\}<
\min\left\{\tfrac{2}{5}\left(2\delta+\tfrac{1}{r}\right),\tfrac{2}{9}\left(4\delta+\tfrac{1}{r}\right)\right\}=
\tfrac{2}{9}\left(4\delta+\tfrac{1}{r}\right)\quad\forall \delta\in\left(\tfrac{11}{32},\tfrac{1}{2}\right)\, ,
\quad\forall r\in[2,4).
$$
Therefore \eq{startup} becomes
$$
\frac{3}{8}<\frac{8\delta+1}{10}<\gamma<\min\left\{\frac{4\delta+1}{6},2\left(2\delta-1+\frac{1}{r}\right)\right\}
$$
which is the open triangle $A_r$ (contained in $A_2$) having vertices
$$
\left(\frac{21}{32}-\frac{5}{8r},\frac{5}{8}-\frac{1}{2r}\right)\, ,\qquad
\left(\frac{13}{20}-\frac{3}{5r},\frac{3}{5}-\frac{2}{5r}\right)\, ,\qquad\left(\frac12,\frac12\right)\, ;
$$
this triangle $A_r$ coincides with $A_2$ when $r=2$, it is nonemepty as long as $r<4$ and its closure collapses to the point $(1/2,1/2)$ as $r\to4$.
Hence, for $r\in[2,4)$ we may find $(\delta,\gamma)\in A_r$ for which both $V$ and $f$ satisfy all the stated properties, except for
\eq{blupofL3}$_2$: to obtain this last property, by \eq{VLqq} one may take $(\delta,\gamma)$ sufficiently close to the vertex
$(13/20-3/5r,3/5-2/5r)$.\par
We finally extend the framework of Theorem \ref{main} and of the first part of the present proof to the case of a countable number of uniformly
bounded blow-up instants. To this end, we slightly change notation. Let $V_{T,\delta,\gamma}(\xi,t)$ be the vector field in \eq{scelta} for some
$T>0$ and some $(\delta,\gamma)\in A_r$. For any $\xi_0\in\R^3$ we can repeat the construction leading to Theorem \ref{main}
by starting with $V_{T,\delta,\gamma}(\xi-\xi_0,t)$, as for the ``shifted bubble'' \eq{bubble} below. Therefore, the solution in Theorem
\ref{main} has six degrees of freedom described by the six parameters:
\neweq{sixparameters}
\xi_0\in\R^3\, ,\quad(\delta,\gamma)\in A_r\, ,\quad T>0\, .
\endeq
For simplicity, we maintain the first five parameters in \eq{sixparameters} fixed and we only focus on different
blow-up instants $T_m$, all occurring at the origin ($\xi_0=0$).\par
Let $(\delta,\gamma)\in A_r$ and take any bounded strictly increasing sequence $\{T_m\}$ of positive instants:
$0<T_1<T_2<...<T_m<T$ for some $T>0$. Let $V$ be as in \eq{scelta}, put $V_m(\xi,t)=V_{T_m,\delta,\gamma}(\xi,t)$ and take
$$
W(\xi,t)=\sum_{m=1}^{\infty}\frac{V_m(\xi,t)}{2^m}\, ,
\quad
W_0(\xi)=W(\xi,0)=\sum_{m=1}^{\infty}\frac{1}{2^m}\left(\tfrac{2yz\, T_m^{2\delta}}{\left[T_m^{2\gamma}+|\xi|^2\right]^{5/2}},
\tfrac{-xz\, T_m^{2\delta}}{\left[T_m^{2\gamma}+|\xi|^2\right]^{5/2}},\tfrac{-xy\, T_m^{2\delta}}{\left[T_m^{2\gamma}+|\xi|^2\right]^{5/2}}\right)\in\G_\sigma ,
$$
see \eq{spaceG}. Therefore, by extending each $V_m$ over $(T_m,T]$,
$$
\|W\|_{L^\infty(0,T;L^2(\R^3))}\le\sum_{m=1}^{\infty}\frac{\|V_m\|_{L^\infty(0,T;L^2(\R^3))}}{2^m}=\|V\|_{L^\infty(0,T;L^2(\R^3))}
$$
and similarly for the other norms. The delicate point is to bound $\|(W\cdot\nabla)W\|_{L^2(0,T;L^{6/5}(\R^3))}$. To this end, we notice that
$$(W\cdot\nabla)W=\sum_{i,j=1}^{\infty}\frac{(V_i\cdot\nabla)V_j}{2^{i+j}}$$
and this series is composed by two kinds of terms.\par\noindent
$\bullet$ If $j=i$ then each $(V_i\cdot\nabla)V_i$ enjoys the same regularity properties as in Theorem \ref{main} and, hence,
$$
\sum_{i=1}^{\infty}\frac{(V_i\cdot\nabla)V_i}{4^i}\in L^2_{\rm loc}(\R_+;L^{6/5}(\R^3))\, .
$$
$\bullet$ If $j\neq i$ then $V_j$ and $V_i$ are regular, respectively, when $t=T_i$ and $t=T_j$: since $T_j\neq T_i$ we then have
$$
\sum_{\stackrel{i,j=1}{j\neq i}}^{\infty}\frac{(V_i\cdot\nabla)V_j}{2^{i+j}}\in L^2_{\rm loc}(\R_+;L^{6/5}(\R^3))\, .
$$

Next, let $P=P_T$ be as in \eq{pressure} and put
$$
Q(\xi,t)=\sum_{m=1}^{\infty}\frac{P_{T_m}(\xi,t)}{2^m}\, .
$$
Since $V_m$ and $P_{T_m}$ have a singularity at $t=T_m$ and since they can be extended also for $t>T_m$, see the proof of Theorem \ref{main},
all the statements can be obtained in the same way. Finally, put again
$$
f:=W_t-\Delta W+(W\cdot\nabla)W+\nabla Q\in L^1(0,T;L^2(\R^3))\cap L^2(0,T;L^{6/5}(\R^3))
$$
and the proof of Theorem \ref{main2} is complete, up to renaming $W$ with $V$.
\end{proof}

\subsection{Proof of Theorem \ref{main3}}\label{prooftheomain3}

Here we take $f$ to satisfy both the {\em second minimal requests} in Definition \ref{LHsolutions} and Proposition \ref{summary}:
$f\in L^{5/4}_{\rm loc}(\R_+;L^2(\R^3))$ as a starting assumption, with the time-exponent free to increase until some maximum threshold.
We use again the computations in Section \ref{reconstruction} and we repeat the steps of the proof of Theorem \ref{main}.
We first prove Theorem \ref{main3} for solutions to \eqref{NS-intro} that blow up in a unique instant (and point), then
we go to a countable number of points, finally we increase the time exponent.\par
By replacing $A$ in \eq{cns} with $B$ in \eq{cns2}, we can repeat the proof of Theorem \ref{main}, to obtain

\begin{theorem}\label{manna}
Let $\G_\sigma$ be as in \eqref{spaceG}, let $B$ be as in \eqref{cns}. For any $(\delta,\gamma)\in B$ and $T>0$ there exist
$$ V_0\in\G_\sigma\, ,\qquad f\in L^{5/4}_{\rm loc}(\R_+;L^2(\R^3))\subset L^1_{\rm loc}(\R_+;L^2(\R^3))\, ,$$
such that $T\mapsto\|\nabla V_0\|_{L^2(\R^3)}$ is decreasing and \eqref{NS-intro} admits a unique global Leray-Hopf solution
$V\in C^0(\R_+;L^2(\R^3))\cap L^2_{\rm loc}(\R_+;H^1(\R^3))$ which satisfies the energy equality, \eqref{energyineq} with equality sign, for all
$t>0$; moreover, $P\in L^2_{\rm loc}(\R_+;L^p(\R^3))$ for all $p\in[1,\infty)$, $\nabla P\in L^{5/4}_{\rm loc}(\R_+;L^2(\R^3))$ and
$$\lim_{t\to T}\|V(t)\|_{L^3(\R^3)}=\lim_{t\to T}\|\nabla V(t)\|_{L^2(\R^3)}=\infty,\quad
\exists\overline{r}=\overline{r}(\delta,\gamma)>\frac54\mbox{ s.t. }\lim_{t\to T}\|\nabla P\|_{L^r(0,t;L^2(\R^3))}=\infty
\ \forall r>\overline{r}\, ;$$
such solution satisfies $V,P\in C^\infty(\R^3\times\R_+\setminus\{(0,T)\})$.
\end{theorem}

By using Theorem \ref{manna} and by arguing as in the proof of Theorem \ref{main2}, we proceed with the full proof.

\begin{proof}[Proof of Theorem \ref{main3}] We consider the shrinked triangle
$$
B_{5/4}:=\left\{(\delta,\gamma)\in B;\, \frac{8\delta+1}{10}<\gamma\right\}
$$
which is nonempty and has vertices at $(\delta,\gamma)\in\{(7/16,9/20);(9/20,7/15);(1/2,1/2)\}$; notice the {\em strict} lower bound.
Since $(\delta,\gamma)\in B_{5/4}$ implies $\delta<1/2$, taking $(\delta,\gamma)\in B_{5/4}$ yields the blow-up conditions in both \eq{blupgrad4}
and \eq{lastbutnotleast}, that is, the solution determined in Theorem \ref{manna} also satisfies
$$
\lim_{t\to T}\, \sqrt[4]{T-t}\, \|\nabla V(t)\|_{L^2(\R^3)}=\infty\, ,\quad V\not\in L^8(0,T;L^4(\R^3))\, ,\quad
\nabla V\not\in L^4(0,T;L^2(\R^3))\, ,
$$
that are three of the four blow-up conditions in Theorem \ref{main3} when a unique blow-up instant $T>0$ exists.
Only the last blow-up condition \eq{blupofL3n2}$_2$ is missing.\par
Before doing this, let us increase $r$ up to $r<4/3$ and let us rewrite \eq{sumg22} in the particular case $q=2$:
$$
g\in L^r(0,T;L^2(\R^3))\ \Longleftrightarrow\ \gamma<\min\big\{
\tfrac{2}{7}(2\delta+\tfrac{1}{r}),\tfrac{2}{11}(4\delta+\tfrac{1}{r}),\, \tfrac{2}{3}(2\delta-\tfrac{r-1}{r})\big\}\, .
$$
Taking $r\ge5/4$, the triangle $B_{5/4}$ should then be further shrink and replaced by the (more stringent) conditions
\neweq{startup2}
\frac{9}{20}<\frac{8\delta+1}{10}<\gamma<\min\left\{\frac{4\delta}{3},\frac{4\delta+1}{6},\frac{2}{7}\left(2\delta+\frac{1}{r}\right),
\frac{2}{11}\left(4\delta+\frac{1}{r}\right),\frac{2}{3}\left(2\delta-1+\frac{1}{r}\right),\frac12\right\}\, .
\endeq
Since $\delta>7/16$, the first term within the min in \eq{startup2} can be dropped. Also notice that
$$
\min\left\{\tfrac{4\delta+1}{6},\tfrac{2}{3}\left(2\delta-1+\tfrac{1}{r}\right)\right\}<
\min\left\{\tfrac{2}{7}\left(2\delta+\tfrac{1}{r}\right),\tfrac{2}{11}\left(4\delta+\tfrac{1}{r}\right)\right\}=
\tfrac{2}{11}\left(4\delta+\tfrac{1}{r}\right)\quad\forall \delta\in\left(\tfrac{7}{16},\tfrac{1}{2}\right)\, ,
\quad\forall r\in\left[\tfrac54,\tfrac43\right).
$$
Therefore \eq{startup2} becomes
$$
\frac{8\delta+1}{10}<\gamma<\min\left\{\frac{4\delta+1}{6},\frac{2}{3}\left(2\delta-1+\frac{1}{r}\right)\right\}
$$
which is the open triangle $B_r$ (contained in $B_{5/4}$) having vertices
$$
\left(\frac{23}{16}-\frac{5}{4r},\frac{5}{4}-\frac{1}{r}\right)\, ,\qquad
\left(\frac{5}{4}-\frac{1}{r},1-\frac{2}{3r}\right)\, ,\qquad\left(\frac12,\frac12\right)\, ;
$$
this triangle $B_r$ coincides with $B_{5/4}$ when $r=5/4$, it is nonemepty as long as $r<4/3$ and its closure collapses to the point $(1/2,1/2)$ as $r\to4/3$.
Hence, for $r\in[5/4,4/3)$ we may find $(\delta,\gamma)\in B_r$ for which both $V$ and $f$ satisfy all the stated properties, except for
\eq{blupofL3n2}$_2$: to obtain this last property, by \eq{VLqq} one may take $(\delta,\gamma)$ sufficiently close to the vertex
$(13/20-3/5r,3/5-2/5r)$.\par
Finally, the extension of the result to a countable number of blow-up instants, can be obtained by combining
the just improved version of Theorem \ref{manna} with the same procedure (convergent series) as in the proof of Theorem \ref{main2}.
\end{proof}

\subsection{Proof of Theorem \ref{inflating}}\label{sharpsohr}

Proposition \ref{minimalf} requires that $V_0\in H^1(\R^3)$, that holds in our construction, see \eq{V0explicit} and the subsequent comments.
But Proposition \ref{minimalf} also requires that $f\in L^2_{\rm loc}(\R_+;L^2(\R^3))$ and we preliminarily show that, even if instead of
\eq{blupL3} we assume the ``weaker blow-up condition'' \eq{blupenstrophy},
\neweq{noL2L2}
\mbox{the examples of Theorems \ref{main2} and \ref{main3} cannot be extended to the case where }f\in L^2_{\rm loc}(0,T;L^2(\R^3))\, .
\endeq

Since it is visually enlightening, we give a geometric evidence to \eq{noL2L2} before entering in the details of the proof of Theorem
\ref{inflating}. We start by discussing the exponents $q>1$ and $r\ge1$ for which Theorems \ref{main} and \ref{manna} continue to hold when
either $f\in L^2_{\rm loc}(\R_+;L^q(\R^3))$ or $f\in L^r_{\rm loc}(\R_+;L^2(\R^3))$,
that is, when the $L^2$-integrability of $f$ is maintained either with respect to time or with respect to space.
After replacing \eq{blupL3} with \eq{blupenstrophy}, the statements of
Lemmas \ref{existencecond}-\ref{allterms}-\ref{wekandblup} all hold whenever
\neweq{capital2}
(\delta,\gamma)\in D:=\left\{(\delta,\gamma)\in\R^2_+;\, \frac{4}{5}\delta<\gamma<\min\left[\frac43 \delta,\, \frac{4\delta+1}{6}
\, ,1\right]\right\}\, .
\endeq
Hence, the triangles $A$ and $B$ in \eq{cns} and \eq{cns2} are ``inflated'', see Figure \ref{inflated}.
\begin{figure}[ht]
\begin{center}
\includegraphics[height=4cm]{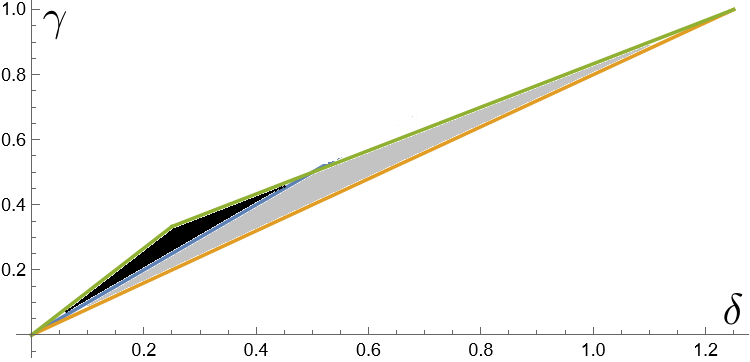}
\end{center}\vskip-5mm
\caption{The triangle $A$ in \eq{cns} (black) and the inflated triangle $D$ in \eq{capital2} (black+grey).}\label{inflated}
\end{figure}

Next, we extend Lemma \ref{defforce2}. Let $V$ be as in \eqref{scelta} and $g$ be as in \eqref{finalf}. We particularise \eq{sumg22}
to the cases when $g$ is square-integrable either with respect to time or with respect to space:
\neweq{unosquare}
\begin{array}{rcl}
\forall q>1\quad g\in L^2(0,T;L^q(\R^3)) &\Longleftrightarrow& \gamma<\min\big\{
\tfrac{q}{5q-3}(2\delta+\tfrac{1}{2}),\, \tfrac{q}{7q-3}(4\delta+\tfrac{1}{2}),\, \tfrac{q}{3q-3}(2\delta-\tfrac{1}{2})\big\}\, ,\\
\forall r\ge1\quad g\in L^r(0,T;L^2(\R^3)) &\Longleftrightarrow& \gamma<\min\big\{
\tfrac{2}{7}(2\delta+\tfrac{1}{r}),\, \tfrac{2}{11}(4\delta+\tfrac{1}{r}),\, \tfrac{2}{3}(2\delta-\tfrac{r-1}{r})\big\}\, .
\end{array}
\endeq

Then we introduce the pressure $P$ as in \eq{pressure}, by following the integrability of $(V\cdot\nabla)V$
and, after setting again $f:=g+\nabla P$, by \eq{nablaP} and \eq{unosquare} we infer that
\neweq{tresquare}
\begin{array}{rcl}
\forall q>1\quad f\in L^2(0,T;L^q(\R^3)) &\Longleftrightarrow& \gamma<\min\big\{
\tfrac{q}{5q-3}(2\delta+\tfrac{1}{2}),\, \tfrac{q}{7q-3}(4\delta+\tfrac{1}{2}),\, \tfrac{q}{3q-3}(2\delta-\tfrac{1}{2})\big\}\, ,\\
\forall r\ge1\quad f\in L^r(0,T;L^2(\R^3)) &\Longleftrightarrow& \gamma<\min\big\{
\tfrac{2}{7}(2\delta+\tfrac{1}{r}),\, \tfrac{2}{11}(4\delta+\tfrac{1}{r}),\, \tfrac{2}{3}(2\delta-\tfrac{r-1}{r})\big\}\, .
\end{array}
\endeq
The question then reduces to find for which $q\in(1,2]$ (resp.\ $r\in[1,2]$) the triangle defined in \eq{capital2} intersects the two polygons
(either a triangle or a quadrilateral) defined by \eq{tresquare}, that is,
\neweq{polygons}
\begin{array}{rcl}
0<\frac{4}{5}\delta<\gamma<\Psi_q(\delta):=\min\big\{\tfrac{q}{5q-3}(2\delta+\tfrac{1}{2}),\tfrac{q}{7q-3}(4\delta+\tfrac{1}{2}), \tfrac{q}{3q-3}(2\delta-\tfrac{1}{2}),\tfrac{4\delta+1}{6},\frac43 \delta,1\big\}\, ,\\
0<\frac{4}{5}\delta<\gamma<\Phi_r(\delta):=
\min\big\{\tfrac{2}{7}(2\delta+\tfrac{1}{r}),\tfrac{2}{11}(4\delta+\tfrac{1}{r}),\tfrac{2}{3}(2\delta-\tfrac{r-1}{r}),
\tfrac{4\delta+1}{6},\tfrac43 \delta,1\big\}\, .
\end{array}
\endeq
In the limit case $q=r=2$, we have $\Phi_2\equiv\Psi_2$; we are now ready to graphically explain how to deduce \eq{noL2L2}.\par
By fixing $r=2$ and letting $q\in(1,2]$ vary, we obtain the pictures in Figure \ref{varyq}. On the left ($q=3/2$), the polygon
in \eq{polygons}$_1$ intersects the triangle in \eq{capital2}. On the middle ($q=5/3$), the polygon in \eq{polygons}$_1$
intersects the triangle in \eq{capital2}. On the right ($q=2$), the polygon in \eq{polygons}$_1$ no longer intersects the triangle in \eq{capital2}.
\begin{figure}[ht]
\begin{center}
\includegraphics[height=27mm]{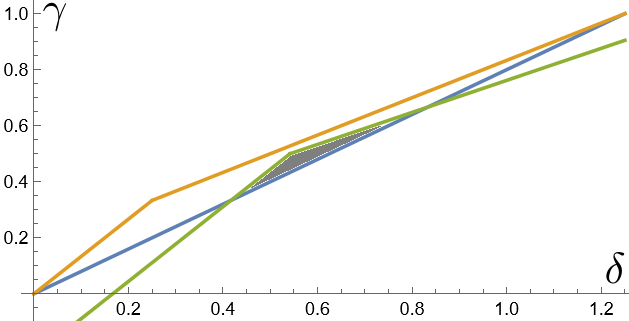}\quad\includegraphics[height=27mm]{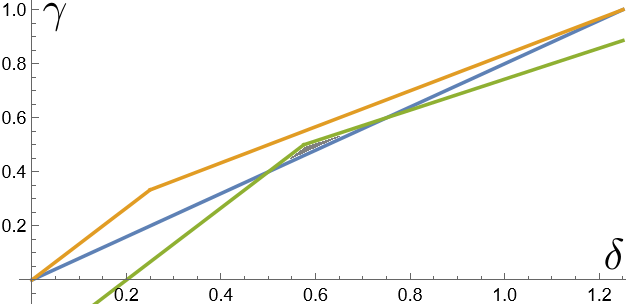}\quad\includegraphics[height=27mm]{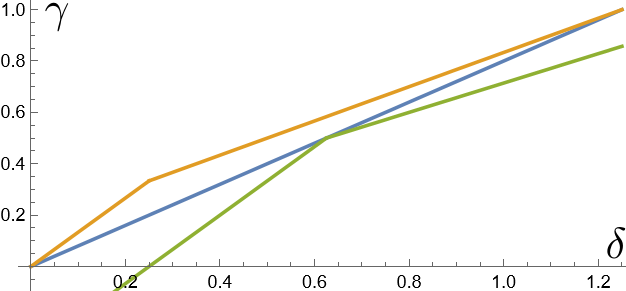}
\end{center}\vskip-5mm
\caption{The triangle in \eq{capital2} and graphs of the functions $\Psi_{3/2}$ (left), $\Psi_{5/3}$ (middle), $\Psi_2$ (right).}\label{varyq}
\end{figure}

Similarly, by fixing $q=2$ and letting $r\in[1,2]$ vary, we obtain the pictures in Figure \ref{varyr}. On the left ($r=8/5$), the polygon
defined by \eq{polygons}$_2$ intersects the triangle in \eq{capital2}. On the middle ($r=9/5$), the polygon in \eq{polygons}$_2$
intersects the triangle in \eq{capital2}. On the right ($r=2$), the polygon in \eq{polygons}$_2$ no longer intersects the triangle
in \eq{capital2}.
\begin{figure}[ht]
\begin{center}
\includegraphics[height=25mm]{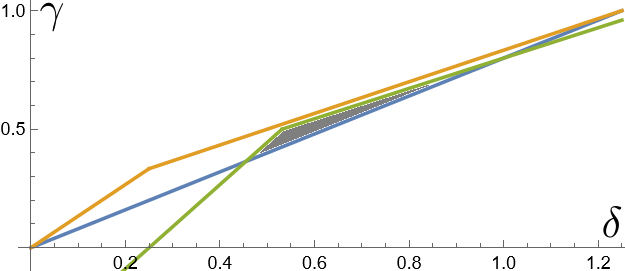}\quad\includegraphics[height=25mm]{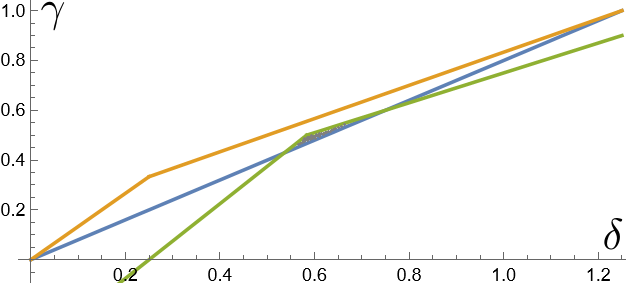}\quad\includegraphics[height=25mm]{fig_q2.png}
\end{center}\vskip-5mm
\caption{The triangle in \eq{capital2} and graphs of the functions $\Phi_{8/5}$ (left), $\Phi_{9/5}$ (middle), $\Phi_2$ (right).}\label{varyr}
\end{figure}

Therefore, the polygons in \eq{polygons} intersect the triangle in \eq{capital2} if and only if $(q,r)\in(1,2]\times[1,2]\setminus\{(2,2)\}$
and this illustrates \eq{noL2L2}.

\begin{proof}[Proof of Theorem \ref{inflating}] Except for the blow-up rates in \eq{quantitative}-\eq{quantitative2}, all the statements in Theorem \ref{inflating} may be obtained as in the proof of Theorem \ref{main2} (including the countable number of blow-up instants),
provided that $(\delta,\gamma)$ satisfy the conditions in \eq{polygons}. So, let us now prove \eq{quantitative}-\eq{quantitative2}.\par
$\diamondsuit$ When $r=2$, for the region \eq{polygons}$_1$ to be nonempty it is necessary that
$$
\frac{4\delta}{5}<\frac{q}{3q-3}\left(2\delta-\frac{1}{2}\right)\ \Longleftrightarrow\ \delta>\frac{5q}{4(6-q)}\, ;
$$
since $\tfrac{5q}{4(6-q)}\ge\tfrac14$ for all $q\in[1,2]$, this implies that $8\delta\ge4\delta+1$ and that
$4\delta/3$ can be omitted in the min inside \eq{polygons}$_1$. Moreover, $4\delta+1<6$ because of the constraint $\delta<5/4$:
we may also omit 1. Then \eq{polygons}$_1$ may be rewritten as
\neweq{conditionr2}
\frac{q}{6-q}<\frac{4\delta}{5}<\gamma<\min\left\{\frac{q}{5q-3}\left(2\delta+\frac{1}{2}\right),\frac{q}{7q-3}\left(4\delta+\frac{1}{2}\right), \frac{q}{3q-3}\left(2\delta-\frac{1}{2}\right),\frac{4\delta+1}{6}\right\}\, .
\endeq
There are still too many terms (four) inside the min. So, we consider the simpler situation
where $q\in[\tfrac32,2)$: in such case, a straightforward observation reads
$$
\frac{q}{5q-3}\left(2\delta+\frac{1}{2}\right)=\frac{q}{2(5q-3)}(4\delta+1)\ \stackrel{q\ge\tfrac32}{\le}\ \frac{4\delta+1}{6}\, ,
$$
showing that the last term in the min in \eq{conditionr2} can be omitted. The next step is to eliminate the second term for which
we need to show that
\neweq{claim01}
\min\left\{\frac{q}{5q-3}\left(2\delta+\frac{1}{2}\right),\frac{q}{3q-3}\left(2\delta-\frac{1}{2}\right)\right\}\le
\frac{q}{7q-3}\left(4\delta+\frac{1}{2}\right)\, .
\endeq
For this, we preliminarily notice three facts
$$
\tfrac{q}{5q-3}\left(2\delta+\tfrac{1}{2}\right)\le\tfrac{q}{7q-3}\left(4\delta+\tfrac{1}{2}\right)\, \Leftrightarrow\,
\delta\ge\tfrac{q}{6(q-1)}\, ,\qquad
\tfrac{q}{3q-3}\left(2\delta-\tfrac{1}{2}\right)\le\tfrac{q}{7q-3}\left(4\delta+\tfrac{1}{2}\right)\, \Leftrightarrow\,
\delta\le\tfrac{5q-3}{2(q+3)}\, ,
$$
$$
\frac{q}{6(q-1)}\le1-\frac{3}{4q}\le\frac{5q-3}{2(q+3)}\qquad\forall q\in\left[\tfrac32,2\right)\, ;
$$
they yield the following implications
$$
\delta\le1-\tfrac{3}{4q}\, \Rightarrow\, \delta\le\tfrac{5q-3}{2(q+3)}\, \Leftrightarrow\,
\tfrac{q}{3q-3}\left(2\delta-\tfrac{1}{2}\right)\le\tfrac{q}{7q-3}\left(4\delta+\tfrac{1}{2}\right)\, ,
$$
$$
\delta\ge1-\tfrac{3}{4q}\, \Rightarrow\, \delta\ge\tfrac{q}{6(q-1)}\, \Leftrightarrow\,
\tfrac{q}{5q-3}\left(2\delta+\tfrac{1}{2}\right)\le\tfrac{q}{7q-3}\left(4\delta+\tfrac{1}{2}\right)\, .
$$
Hence, in any case, \eq{claim01} follows. In turn, this shows that \eq{conditionr2} reduces to
$$
\frac{q}{6-q}<\frac{4\delta}{5}<\gamma<\min\left\{\frac{q}{5q-3}\left(2\delta+\frac{1}{2}\right), \frac{q}{3q-3}\left(2\delta-\frac{1}{2}\right)\right\}\, ,
$$
which defines a triangle whose vertices are (recall that $q\ge3/2>6/5$)
$$
(\delta,\gamma)\in\left\{\left(\frac{5q}{4(6-q)},\frac{q}{6-q}\right);\left(1-\frac{3}{4q},\frac12\right);
\left(\frac{5q}{4(5q-6)},\frac{q}{5q-6}\right)\right\}\, .
$$
The important vertex is the only one not belonging to the line $5\gamma=4\delta$: its distance from the line will give the blow-up
rate for $\|\nabla V(t)\|_{L^2(\R^3)}$ as $t\to T_m$. From \eq{nablaV2} we know that
\neweq{weknow}
|T_m-t|^\lambda\|\nabla V(t)\|_{L^2(\R^3)}\sim|T_m-t|^{\lambda+2\delta-5\gamma/2}\qquad\mbox{as }t\to T_m\, .
\endeq
The line parallel to $5\gamma=4\delta$ and containing the middle vertex has equation
$$
\frac{5\gamma}{2}=2\delta+\frac{3(2-q)}{4q}
$$
and, then, \eq{quantitative} holds for all $\lambda<\tfrac{3(2-q)}{4q}$.
\par\medskip

$\diamondsuit$ When $q=2$ is fixed, we first notice that
$$
\frac{4\delta}{3}\ge\frac{2}{3}\left(2\delta-1+\frac{1}{r}\right)\qquad\forall r\in[1,2]
$$
and, hence, $4\delta/3$ can be omitted in the min inside \eq{polygons}$_2$. Moreover, $4\delta+1<6$ because of the constraint $\delta<5/4$:
we may also omit 1. For the region \eq{polygons}$_2$ to be nonempty it is necessary that
\neweq{deltainterval}
\frac{4\delta}{5}<\min\left\{\frac{2}{7}\left(2\delta+\frac{1}{r}\right),\frac{2}{3}\left(2\delta-1+\frac{1}{r}\right)\right\}\, \Longleftrightarrow\, \frac{5}{4}\left(1-\frac{1}{r}\right)<\delta<\frac{5}{4r}
\endeq
and, from now on, we take $\delta$ in this interval: its amplitude is $\tfrac{5}{2r}-\tfrac54 >0$ and tends to 0 for $r\to2$, recovering
the behaviour illustrated in Figure \ref{varyr}. Therefore, \eq{polygons}$_2$ is reduced to
\neweq{conditionq2}
1-\frac{1}{r}<\frac{4}{5}\delta<\gamma<\min\left\{\frac{2}{7}\left(2\delta+\frac{1}{r}\right),
\frac{2}{3}\left(2\delta-\frac{r-1}{r}\right),\frac{2}{11}\left(4\delta+\frac{1}{r}\right),\frac{4\delta+1}{6}\right\}\, .
\endeq

This polygon changes shape and number of vertices, depending on $r$. This is the point where it is crucial to assume that $r\ge4/3$.
In such case, we claim that \eq{conditionq2} reduces to
\neweq{conditionq2r43}
1-\frac{1}{r}<\frac{4}{5}\delta<\gamma<\min\left\{\frac{2}{7}\left(2\delta+\frac{1}{r}\right),
\frac{2}{3}\left(2\delta-\frac{r-1}{r}\right)\right\}\, .
\endeq
To this end, we notice that
$$
\frac{2}{3}\left(2\delta-\frac{r-1}{r}\right)\le\frac{2}{7}\left(2\delta+\frac{1}{r}\right)\ \Longrightarrow\
\delta\le\frac{7}{8}-\frac{1}{2r}
$$
and the proof of \eq{conditionq2r43} then follows if we show the following facts:
\begin{align}
\frac{2}{3}\left(2\delta-\frac{r-1}{r}\right)\le\frac{4\delta+1}{6}\quad\mbox{and}\quad\frac{2}{3}\left(2\delta-\frac{r-1}{r}\right)
\le\frac{2}{11}\left(4\delta+\frac{1}{r}\right)\qquad\forall\delta\le\frac{7}{8}-\frac{1}{2r}\, , \label{monstre1}\\
\frac{2}{7}\left(2\delta+\frac{1}{r}\right)\le\frac{4\delta+1}{6}\quad\mbox{and}\quad\frac{2}{7}\left(2\delta+\frac{1}{r}\right)
\le\frac{2}{11}\left(4\delta+\frac{1}{r}\right)\qquad\forall\delta\ge\frac{7}{8}-\frac{1}{2r}\, . \label{monstre2}
\end{align}
The following chain of inequalities proves \eq{monstre1}$_1$:
$$
4\left(2\delta-\frac{r-1}{r}\right)=4\delta+4\delta-4+\frac{4}{r}\ \stackrel{\delta\le\tfrac{7}{8}-\tfrac{1}{2r}}{\le}\
4\delta+\frac{2}{r}-\frac{1}{2}
\ \stackrel{r\ge\tfrac43}{\le}\ 4\delta+1\, .
$$
The following chain of inequalities proves \eq{monstre1}$_2$:
$$
11\left(2\delta-\frac{r-1}{r}\right)=12\delta+10\delta-11+\frac{11}{r}\ \stackrel{\delta\le\tfrac{7}{8}-\tfrac{1}{2r}}{\le}\
12\delta-\frac{4}{3}+\frac{6}{r}
\ \stackrel{r\ge\tfrac43}{\le}\ 3\left(4\delta+\frac{1}{r}\right)\, .
$$
The following chain of inequalities proves \eq{monstre2}$_1$:
$$
12\left(2\delta+\frac{1}{r}\right)=28\delta-4\delta+\frac{12}{r}\ \stackrel{\delta\ge\tfrac{7}{8}-\tfrac{1}{2r}}{\le}\
28\delta+\frac{14}{r}-\frac{7}{2}\ \stackrel{r\ge\tfrac43}{\le}\ 7(4\delta+1)\, .
$$
The following chain of inequalities proves \eq{monstre2}$_2$:
$$
11\left(2\delta+\frac{1}{r}\right)=28\delta-6\delta+\frac{11}{r}\ \stackrel{\delta\ge\tfrac{7}{8}-\tfrac{1}{2r}}{\le}\
28\delta+\frac{14}{r}-\frac{21}{4}\ \stackrel{r\ge\tfrac43}{\le}\ 7\left(4\delta+\frac{1}{r}\right)\, .
$$

All the inequalities in \eq{monstre1}-\eq{monstre2} are proved, which also proves \eq{conditionq2r43}.
The region in \eq{conditionq2r43} is an open triangle whose vertices are
$$
(\delta,\gamma)\in\left\{\left(\frac{5}{4}-\frac{5}{4r},1-\frac{1}{r}\right);\left(\frac{7}{8}-\frac{1}{2r},\frac12\right);
\left(\frac{5}{4r},\frac{1}{r}\right)\right\}\, .
$$
The important vertex is the only one not belonging to the line $5\gamma=4\delta$: its distance from the line will give the blow-up rate
\eq{quantitative2}. Using again \eq{weknow}, we find that the line parallel to $5\gamma=4\delta$ and containing the middle vertex has equation
$$
\frac{5\gamma}{2}=2\delta+\frac{2-r}{2r}
$$
and, then, \eq{quantitative2} holds for all $\lambda<\tfrac{2-r}{2r}$.
\end{proof}

\begin{remark}
The parameter $T>0$ in both Theorems \ref{main} and \ref{manna} is relevant not only because it determines the blow-up
time but also because it measures the size of the initial velocity $V_0$ in \eqref{scelta}. Using \eqref{general} for $t=0$ with,
respectively, $(\alpha,k,p,q)=(\delta,2,5,2)$ and $(\alpha,k,p,q)=(\delta,3,7,2)$, we obtain that
$$
\|V_0\|_{L^2(\R^3)}^2=C T^{4\delta-3\gamma}\, ,\qquad\|\nabla V_0\|_{L^2(\R^3)}^2=C T^{4\delta-5\gamma}\, .
$$
Since for any $(\delta,\gamma)\in A\cup B$ we have $3\gamma<4\delta<5\gamma$, we infer that
$$
T\mapsto\|V_0\|_{L^2(\R^3)}\mbox{ is increasing,}\qquad T\mapsto\|\nabla V_0\|_{L^2(\R^3)}\mbox{ is decreasing.}
$$
Also for all $(\delta,\gamma)\in D$, see \eqref{capital2}, the same properties hold.
\end{remark}

\subsection{Proof of Theorem \ref{mainLinfty}}\label{blupLinftyLinfty}

We need to find conditions on $f$ ensuring the {\em pointwise blow-up} of $V$, that is,
\neweq{blupVinfty}
\lim_{t\to T}\|V\|_{L^\infty([0,t)\times\R^3)}=\infty\, .
\endeq
Recall that $V$ in \eq{scelta} is bounded and smooth in $[0,t]\times\R^3$ for all $t<T$. Moreover, from \eq{VLqq} with $q=\infty$,
we know that \eq{blupVinfty} occurs whenever
$$
\gamma>\frac{2\delta}{3}\, .
$$
So, by \eq{sumg22}, the question becomes for which $(r,q)\in[1,\infty)\times(1,\infty)$ there exist $(\delta,\gamma)\in\R_+^2$
such that
$$
\frac{2\delta}{3}<\gamma<\min\left\{\frac43 \delta,\frac{4\delta+1}{6},
\frac{q}{5q-3}\left(2\delta+\frac{1}{r}\right),\frac{q}{7q-3}\left(4\delta+\frac{1}{r}\right),
\frac{q}{3q-3}\left(2\delta-\frac{r-1}{r}\right)\right\}\, .
$$
The first two terms in the min are always strictly larger than $2\delta/3$. Therefore, \eq{blupVinfty} holds for those
$(r,q)$ ensuring the existence of $(\delta,\gamma)\in\R_+^2$ such that
$$
\frac{2\delta}{3}<\gamma<\min\left\{\frac{q}{5q-3}\left(2\delta+\frac{1}{r}\right),\frac{q}{7q-3}\left(4\delta+\frac{1}{r}\right),
\frac{q}{3q-3}\left(2\delta-\frac{r-1}{r}\right)\right\}\, .
$$
We then notice that
$$
\delta\ge\frac{5}{4}-\frac{1}{4}\left(\frac{2}{r}+\frac{3}{q}\right)\ \Longleftrightarrow\ \frac{q}{5q-3}\left(2\delta+\frac{1}{r}\right)\le\frac{q}{3q-3}\left(2\delta-\frac{r-1}{r}\right)\, .
$$
Hence, \eq{blupVinfty} holds for those $(r,q)\in[1,\infty)\times(1,\infty)$ ensuring the existence of
\neweq{unbelievable}
\delta>0\, ,\quad \delta\ge\frac{5}{4}-\frac{1}{4}\left(\frac{2}{r}+\frac{3}{q}\right)
\endeq
such that
$$
\frac{2\delta}{3}<\min\left\{\frac{q}{5q-3}\left(2\delta+\frac{1}{r}\right),\frac{q}{7q-3}\left(4\delta+\frac{1}{r}\right)\right\}\, .
$$
When \eq{unbelievable} hold, we are so left with two inequalities.\par
The {\em first inequality}
\neweq{firstineq}
\frac{2\delta}{3}<\frac{q}{5q-3}\left(2\delta+\frac{1}{r}\right)
\endeq
is satisfied for any $\delta\ge5/4$ (that ensures \eq{unbelievable}) whenever $1<q\le3/2$. If $q>3/2$, two cases may occur:\\
-- if $2/r+3/q\ge5$, then \eq{unbelievable} reduces to $\delta>0$ and \eq{firstineq} is satisfied for $\delta$ sufficiently small whatever
$(r,q)\in[1,\infty)\times(1,\infty)$ are;\\
-- if $2/r+3/q<5$, then \eq{unbelievable} reduces to $\delta\ge\delta^*:=5/4-(2/r+3/q)/4$ and, by taking $\delta=\delta^*$ in \eq{firstineq},
we find that it is satisfied whenever
\neweq{supertech}
1\le r<\frac{2q}{2q-3}\ \Longleftrightarrow\ \mbox{\eqref{ipotesona}$_1$}\, .
\endeq

The {\em second inequality}
\neweq{secondineq}
\frac{2\delta}{3}<\frac{q}{7q-3}\left(4\delta+\frac{1}{r}\right)
\endeq
is satisfied for any $\delta\ge5/4$ (that ensures \eq{unbelievable}) whenever $1<q\le3$. If $q>3$, two cases may occur:\\
-- if $2/r+3/q\ge5$, then \eq{unbelievable} reduces to $\delta>0$ and \eq{secondineq} is satisfied for $\delta$ sufficiently small whatever
$(r,q)\in[1,\infty)\times(1,\infty)$ are;\\
-- if $2/r+3/q<5$, then \eq{unbelievable} reduces to $\delta\ge\delta^*$ and, taking $\delta=\delta^*$ in \eq{secondineq},
we see that it is satisfied when
$$1\le r<\frac{2q(4q-3)}{(5q-3)(q-3)}\, ,$$
which is less restrictive than \eq{supertech}.\par
Summarising, if \eq{ipotesona} holds, then \eq{firstineq} and \eq{secondineq} are fulfilled for some $\delta$ satisfying
\eq{unbelievable}.\par
To prove \eq{bullet} in $t=T$, we notice that, if $(r,q)$ satisfy \eqref{ipotesonagrad}$_1$, then
$$
\frac{q}{5q-3}\left(2\delta+\frac{1}{r}\right)\le\frac{2\delta}{3}\ \Longleftrightarrow\ \frac{3q}{2r(2q-3)}\le\delta,
\qquad\frac{q}{7q-3}\left(4\delta+\frac{1}{r}\right)\le\frac{2\delta}{3}\ \Longleftrightarrow\ \frac{3q}{2r(q-3)}\le\delta.
$$
Therefore, \eq{bullet} follows if we show that there exist $(r,q)$ satisfying \eq{ipotesonagrad}$_1$ such that
\neweq{onemore}
\frac{q}{7q-3}\left(4\delta^*+\frac{1}{r}\right)\le\frac{2\delta^*}{3}\ \Longleftrightarrow\ \frac{2}{r}+\frac{3}{q}\le\frac{5q-6}{4q-3}\in(1,5)
\qquad\forall q>3
\endeq
since, then, we find $\gamma<2\delta^*/3$ and, by \eq{VLqq} (with $q=\infty$) and the same principle as in \eq{continuity}, we infer that
$V\in C^0(\R^3\times\R_+)$ and $V(\xi,T)\equiv0$ in $\R^3$. Condition \eq{onemore} has to be added to \eqref{ipotesonagrad}$_1$ for \eq{bullet}
to hold.\par
The proof of Theorem \ref{mainLinfty} is so complete for a unique blow-up instant and point.
The extension to countable blow-up can be obtained by using the convergent numerical series as in Theorem \ref{main2}.

\subsection{Proof of Theorem \ref{mainLinftygrad}}\label{blupgrad}

By applying Lemma \ref{calculus} to \eq{Vt}-\eq{summable}, we find that
\neweq{VVV}
\begin{array}{rl}
\Delta V\in L^\infty(\R^3\times(0,T))\ \Longleftrightarrow\ \gamma\le\frac{2\delta}{5}\, , &\quad
(V\cdot\nabla)V\in L^\infty(\R^3\times(0,T))\ \Longleftrightarrow\ \gamma\le\frac{4\delta}{7}\, ,\\
\nabla V\in L^\infty(\R^3\times(0,T))\ \Longleftrightarrow\ \gamma\le\frac{\delta}{2}\, ,&\quad
V_t\in L^\infty(\R^3\times(0,T))\ \Longleftrightarrow\ \gamma\le\frac{2\delta-1}{3}\, .
\end{array}
\endeq
Due the {\em magic cancellation} which lead to \eq{summable} and was discussed in Remark \ref{magic}, the condition for $\nabla V$ to
remain bounded is more restrictive than the condition for $(V\cdot\nabla)V$ to remain bounded. \par
We need here to find conditions on $f$ ensuring the pointwise blow-up of $\nabla V$, that is,
\neweq{blupnablaVinfty}
\lim_{t\to T}\|\nabla V\|_{L^\infty([0,t)\times\R^3)}=\infty\, .
\endeq
Recall that $\nabla V$ is bounded and smooth in $[0,t]\times\R^3$ for all $t<T$. Moreover, from \eq{VVV}
we know that \eq{blupnablaVinfty} occurs whenever
$$
\gamma>\frac{\delta}{2}\, .
$$
So, by \eq{sumg22}, the question becomes for which $(r,q)\in[1,\infty)\times(1,\infty)$ there exist $(\delta,\gamma)\in\R_+^2$
such that
$$
\frac{\delta}{2}<\gamma<\min\left\{\frac43 \delta,\frac{4\delta+1}{6},
\frac{q}{5q-3}\left(2\delta+\frac{1}{r}\right),\frac{q}{7q-3}\left(4\delta+\frac{1}{r}\right),
\frac{q}{3q-3}\left(2\delta-\frac{r-1}{r}\right)\right\}\, .
$$
Clearly, the first two terms in the min are always strictly larger than $\delta/2$. Therefore, \eq{blupnablaVinfty} holds for those
$(r,q)$ ensuring the existence of $(\delta,\gamma)\in\R_+^2$ such that
$$
\frac{\delta}{2}<\gamma<\min\left\{\frac{q}{5q-3}\left(2\delta+\frac{1}{r}\right),\frac{q}{7q-3}\left(4\delta+\frac{1}{r}\right),
\frac{q}{3q-3}\left(2\delta-\frac{r-1}{r}\right)\right\}\, .
$$
Arguing as for \eq{unbelievable}, we find that \eq{blupnablaVinfty} holds for those $(r,q)\in[1,\infty)\times(1,\infty)$ ensuring the existence of
$\delta>0$ satisfying \eq{unbelievable} and such that
$$
\frac{\delta}{2}<\min\left\{\frac{q}{5q-3}\left(2\delta+\frac{1}{r}\right),\frac{q}{7q-3}\left(4\delta+\frac{1}{r}\right)\right\}\, .
$$
When \eq{unbelievable} hold, we are so left with two inequalities.\par
The {\em first inequality}
\neweq{firstineqgrad}
\frac{\delta}{2}<\frac{q}{5q-3}\left(2\delta+\frac{1}{r}\right)
\endeq
is satisfied for any $\delta\ge5/4$ (that ensures \eq{unbelievable}) whenever $1<q\le3$. If $q>3$, two cases may occur:\\
-- if $2/r+3/q\ge5$, then \eq{unbelievable} reduces to $\delta>0$ and \eq{firstineqgrad} is satisfied for $\delta$ sufficiently small whatever
$(r,q)\in[1,\infty)\times(1,\infty)$ are;\\
-- if $2/r+3/q<5$, then \eq{unbelievable} reduces to $\delta\ge\delta^*:=5/4-(2/r+3/q)/4$ and, by taking $\delta=\delta^*$ in
\eq{firstineqgrad}, we find that it is satisfied whenever
\neweq{supertechgrad}
1\le r<\frac{2q}{q-3}\ \Longleftrightarrow\ \mbox{\eqref{ipotesonagrad}$_1$}\, .
\endeq

The {\em second inequality}
\neweq{secondineqgrad}
\frac{\delta}{2}<\frac{q}{7q-3}\left(4\delta+\frac{1}{r}\right)
\endeq
is satisfied for any $\delta\ge5/4$ (that ensures \eq{unbelievable}) whatever $(r,q)\in[1,\infty)\times(1,\infty)$ are.\par
Summarising, if \eq{ipotesonagrad} holds, then \eq{firstineqgrad} and \eq{secondineqgrad} are fulfilled for some $\delta$
satisfying \eq{unbelievable}. This proves Theorem \ref{mainLinftygrad} for a unique blow-up instant and point.
The extension to countable blow-up can be obtained by using the convergent numerical series as in Theorem \ref{main2}.

\subsection{Proof of Theorem \ref{final}}\label{Leray_blup}

Let $V$ be as in \eq{scelta}. Then, using \eq{VVV} and arguing as for \eq{iff2}$_3$ and \eq{continuity}, we have that
\neweq{blupsmooth}
\gamma<\min\left\{\frac\delta2,\frac{2\delta-1}{3}\right\}\, \Longrightarrow\, V\in C^1(\R^3\times[0,T]),\quad
V(\xi,T)\equiv\nabla V(\xi,T)\equiv V_t(\xi,T)\equiv0\mbox{ in }\R^3,
\endeq
the last condition being intended as limit for $t\to T$ and extension by continuity. By \eq{VVV} and
aiming to construct an example where {\em only} the second derivatives blow up as $t\to T$, we assume that
\neweq{blupsoloLap}
\frac{2\delta}{5}<\gamma<\min\left\{\frac{\delta}{2}\, ,\, \frac{2\delta-1}{3}\right\}\, .
\endeq
If \eq{blupsoloLap} holds, we necessarily have
\neweq{deltagammalarge}
\delta>\frac{5}{4}\, ,\quad\gamma>\frac{1}{2}
\endeq
and the inequalities $4\delta\ge3\gamma$ and $4\delta+1>6\gamma$ are satisfied so that both the
statements of Lemmas \ref{existencecond}-\ref{allterms} hold. This means that the existence, uniqueness, and energy equality conditions
are all fulfilled. The region in \eq{blupsoloLap} is unbounded, see Figure \ref{LinftyLinfty} (left).
\begin{figure}[ht]
\begin{center}
\includegraphics[height=40mm]{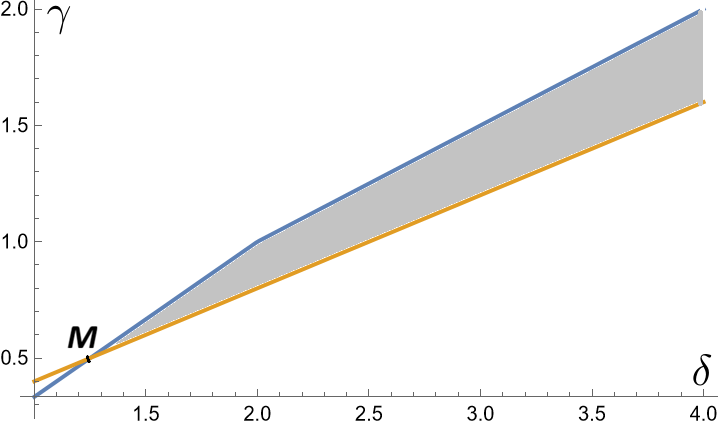}\qquad \includegraphics[height=40mm]{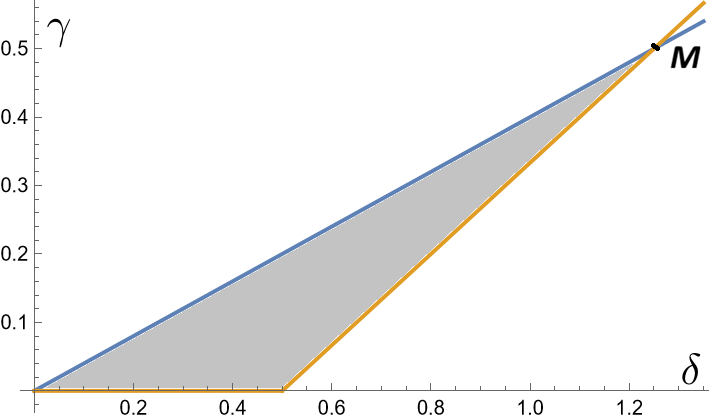}
\end{center}\vskip-5mm
\caption{The point $M(\tfrac54 ,\tfrac12 )$, the region in \eq{blupsoloLap} (left), the triangle in \eq{blupsoloaccel} (right).}\label{LinftyLinfty}
\end{figure}
By taking $(\delta,\gamma)$ inside this region we have $V\in C^1(\R^3\times[0,T])$ while
$\Delta V\not\in L^\infty(\R^3\times(0,T))$. However, for any $q,r<\infty$, if $(\delta,\gamma)$ is sufficiently close to the vertex
$M(\tfrac54 ,\tfrac12 )$, we obtain $\Delta V\in L^r(0,T;L^q(\R^3))$. Let $P$ be as in \eq{pressure} so that $\nabla P$ is given by
\eq{nablaP}, let $f=g+\nabla P$. With slight modifications of the above proofs (in particular, Theorem \ref{main2}),
we conclude the proof of Theorem \ref{final}.

\begin{remark}\label{classical}
The condition \eqref{deltagammalarge} states that both $\delta$ and $\gamma$ are ``large'', as expected if $V$ in \eqref{scelta} has to be smooth with respect to time. To reach the blow-up results in Theorems \ref{main2} and \ref{main3} we need $\delta,\gamma<1/2$, whereas
for Theorem \ref{inflating} we need $\delta<5/4$ and $\gamma<1$.\par
By imposing $2\delta>\max\{5\gamma,3\gamma+1\}$ we obtain classical solutions to \eqref{NS-intro}: to see this, one can use again
\eqref{VVV} combined with \eqref{VLqq} and \eqref{continuity}, then taking $f=g+\nabla p\in C^0(\R^3\times\R_+)$. The statements of Theorems
\ref{mainLinfty}, \ref{mainLinftygrad}, \ref{final} are summarised in the next table.
\begin{center}
\begin{tabular}{|c|c|c|c|c|}
\hline
$2\delta\in$ & $(0,3\gamma)$ & $(3\gamma,4\gamma]$ & $(4\gamma,5\gamma]$ & $(5\gamma,\infty)$ \\
\hline
$t\mapsto\|V(t)\|_{L^\infty(\R^3)}$ & \includegraphics[height=8mm]{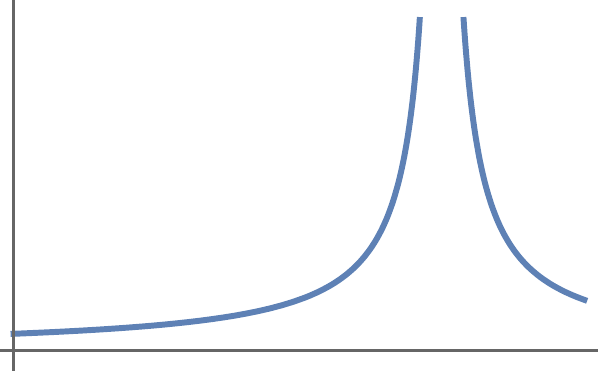} & \includegraphics[height=8mm]{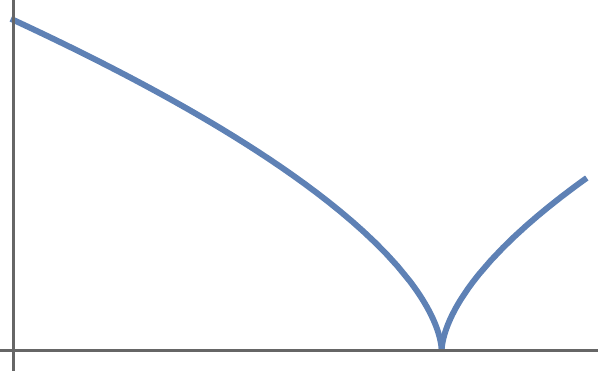} & \includegraphics[height=8mm]{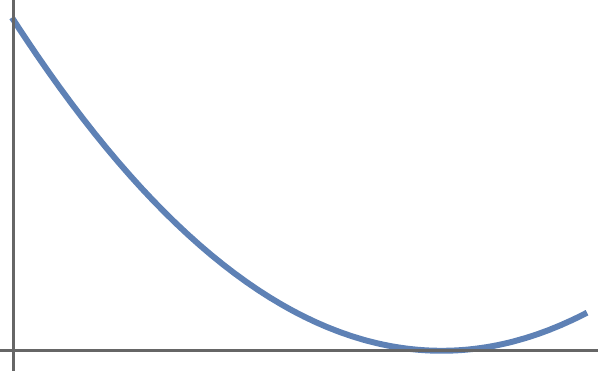} & \includegraphics[height=8mm]{plot3.pdf} \\
\hline
regularity & $V\not\in L^\infty_{\rm loc}(\R^3\times\R_+)$ & $V\in C^0(\R^3\times\R_+)$  & $V\in C^1(\R^3\times\R_+)$ & $V\in C^2(\R^3\times\R_+)$ \\
\hline
annihilation at $t=T$ & none & $V$ & $V,\, \nabla V,\, V_t$ & $V,\, \nabla V,\, V_t,\, \Delta V$ \\
\hline
\end{tabular}
\end{center}

An example of blow-up of the sole acceleration $V_t$ may be obtained by replacing \eqref{blupsmooth} with
\neweq{blupsoloaccel}
\frac{2\delta-1}{3}<\gamma<\frac{2\delta}5\, \Longrightarrow\, V(\xi,T)\equiv\nabla V(\xi,T)\equiv\Delta V(\xi,T)\equiv 0\mbox{ in }\R^3\, :
\endeq
then $\delta<\tfrac{5}{4}$, $\gamma<\tfrac{1}{2}$, and the inequalities $4\delta\ge3\gamma$ and $4\delta+1>6\gamma$ are satisfied: hence,
both the statements of Lemmas \ref{existencecond}-\ref{allterms} hold: existence, uniqueness, and energy equality conditions
are all fulfilled. The region in \eqref{blupsoloaccel} is a triangle having a vertex at the same point $M$, see Figure \ref{LinftyLinfty} (right).
By taking $(\delta,\gamma)$ therein we have $V,\nabla V,\Delta V\in C^0(\R^3\times[0,T])$ and \eqref{blupsoloaccel}
while $V_t\not\in L^\infty(\R^3\times(0,T))$. However, for any $q,r<\infty$, if $(\delta,\gamma)$ is sufficiently close to the vertex
$M$, we obtain $V_t\in L^r(0,T;L^q(\R^3))$. We omit further details.
\end{remark}

\subsection{Proof of Theorem \ref{eulertheo}}

The {\em magic cancellation} leading to \eq{summable} creates a gap between the blow-up of $\nabla\wedge V$ and $(V\cdot\nabla)V$,
see \eq{VVV}. The idea is then to fit $(\delta,\gamma)$ within this gap.\par
Let us rewrite $V=V(\xi,t)$ in \eq{scelta} in spatial spherical coordinates $V=V(|\xi|,\theta,\varphi,t)$:
$$
v^1(|\xi|,\theta,\varphi,t)=\tfrac{g_1(\theta,\varphi)|\xi|^2(T-t)^{2\delta}}{\left[(T-t)^{2\gamma}+|\xi|^2\right]^{5/2}},\quad
v^2(|\xi|,\theta,\varphi,t)=\tfrac{g_2(\theta,\varphi)|\xi|^2(T-t)^{2\delta}}{\left[(T-t)^{2\gamma}+|\xi|^2\right]^{5/2}},\quad
v^3(|\xi|,\theta,\varphi,t)=\tfrac{g_3(\theta,\varphi)|\xi|^2(T-t)^{2\delta}}{\left[(T-t)^{2\gamma}+|\xi|^2\right]^{5/2}},
$$
for some smooth bounded functions $g_i\in C^\infty(\R^2)$ ($i=1,2,3$). If we put $\X=|\xi|\ge0$ and $\Y=(T-t)^\gamma\ge0$, they may be rewritten as
$$
v^1(\X,\Y)=\frac{g_1(\theta,\varphi)\X^2\Y^{2\delta/\gamma}}{\left[\Y^2+\X^2\right]^{5/2}},\quad
v^2(\X,\Y)=\frac{g_2(\theta,\varphi)\X^2\Y^{2\delta/\gamma}}{\left[\Y^2+\X^2\right]^{5/2}},\quad
v^3(\X,\Y)=\frac{g_3(\theta,\varphi)\X^2\Y^{2\delta/\gamma}}{\left[\Y^2+\X^2\right]^{5/2}}.
$$
To check the regularity, we argue in the $(\X,\Y)$-plane, where these functions are in $C^\infty(\R_+^2)$ and we only need to analyse
their behaviour at the origin $(\X,\Y)=(0,0)$ where we set $v^i(0,0)=0$.
We transform these functions in polar coordinates $(\rho,\beta)$ with $0\le\beta\le\tfrac\pi2$ and obtain
$$
v^i(\rho,\beta)=g_i(\theta,\varphi)\, (\cos\beta)^2\, (\sin\beta)^{2\delta/\gamma}\, \rho^{2\delta/\gamma-3},\qquad(i=1,2,3).
$$
Therefore,
\neweq{C0a}
V\in C^{0,\alpha}(\R^3\times[0,T])\ \Longleftrightarrow\ \gamma<\frac{2\delta}{3+\alpha}\, ,
\endeq
while $V\in C^\infty(\R^3\times[0,T))$ because, for $t<T$ the denominator of $V$ does not vanish and neither does $(T-t)$.\par
We apply the same method to $V_t$ in \eq{Vt} and to $(V\cdot\nabla)V$ in \eq{summable} and we obtain
$$
[(V\cdot\nabla)V]^i(\rho,\beta)=\Lambda_i(\theta,\varphi,\beta)\rho^{4\delta/\gamma-7}\, ,\quad
V_t^i(\rho,\beta)=\Gamma_i(\theta,\varphi,\beta)\rho^{(2\delta-1)/\gamma-3}
$$
for some smooth $\Lambda_i,\Gamma_i$ ($i=1,2,3$). Therefore,
$$
(V\cdot\nabla)V\in C^{0,\alpha}(\R^3\times[0,T])\ \Longleftrightarrow\ \gamma<\frac{4\delta}{7+\alpha}\, ,\quad
V_t\in C^{0,\alpha}(\R^3\times[0,T])\ \Longleftrightarrow\ \gamma<\frac{2\delta-1}{3+\alpha}\, .
$$
Then we take $P$ as in \eq{pressure}, so that $\nabla P$ and $(V\cdot\nabla)V$ have the same regularity, see \eq{nablaP}; we put $f=V_t+(V\cdot\nabla)V+\nabla P$
and obtain
$$
f\in C^{0,\alpha}(\R^3\times[0,T])\ \Longleftrightarrow\ \gamma<\min\left\{\frac{4\delta}{7+\alpha},\frac{2\delta-1}{3+\alpha}\right\}\, .
$$

On the other hand, from \eq{VVV} we know that $\nabla V\not\in L^\infty(\R^3\times(0,T))$ if and only if
$\gamma>\tfrac{\delta}{2}$. But we need to refine this result, because the statement is stronger, it requires that the blow-up
of the {\em vorticity} occurs at an $L^1(0,T)$-rate. For $V=(v^1,v^2,v^3)$ as in \eq{scelta}, the relevant partial derivatives
are given by
$$
v^1_y(x,y,z,t)=\frac{2z(T-t)^{2\delta}}{\left[(T-t)^{2\gamma}+|\xi|^2\right]^{5/2}}-
\frac{10y^2z(T-t)^{2\delta}}{\left[(T-t)^{2\gamma}+|\xi|^2\right]^{7/2}},\quad v^1_z(x,y,z,t)=v^1_y(x,z,y,t),
$$
$$
v^2_x(x,y,z,t)=\frac{-z(T-t)^{2\delta}}{\left[(T-t)^{2\gamma}+|\xi|^2\right]^{5/2}}+
\frac{5x^2z(T-t)^{2\delta}}{\left[(T-t)^{2\gamma}+|\xi|^2\right]^{7/2}},\quad v^2_z(x,y,z,t)=v^2_x(z,y,x,t),
$$
$$
v^3_x(x,y,z,t)=\frac{-y(T-t)^{2\delta}}{\left[(T-t)^{2\gamma}+|\xi|^2\right]^{5/2}}+
\frac{5x^2z(T-t)^{2\delta}}{\left[(T-t)^{2\gamma}+|\xi|^2\right]^{7/2}},\quad v^3_y(x,y,z,t)=v^2_x(y,x,z,t).
$$
In particular,
$$
\big|\nabla\wedge V(\xi,t)\big|\ge\left|\frac{3z(T-t)^{2\delta}}{\left[(T-t)^{2\gamma}+|\xi|^2\right]^{5/2}}-
\frac{15y^2z(T-t)^{2\delta}}{\left[(T-t)^{2\gamma}+|\xi|^2\right]^{7/2}}\right|
$$
and, by extending \eq{iff2}$_1$ to the case $q=\infty$, we obtain
$$
\lim_{t\to T}\int_{0}^{t}\|\nabla\wedge V(s)\|_{L^\infty(\R^3)}\, ds=\infty\ \Longleftrightarrow\ \gamma\ge\frac{2\delta+1}{4}\, .
$$
Therefore, the result is proved if there exist $\alpha\in(0,1)$ and $\gamma>0$ such that
$$
\frac{2\delta+1}{4}\le\gamma<\min\left\{\frac{4\delta}{7+\alpha},\frac{2\delta-1}{3+\alpha}\right\}\, .
$$
By taking $\delta$ sufficiently large, one can check that, for any $\alpha\in(0,1)$ there exist $(\delta,\gamma)$ satisfying these constraints.

\section{Appendices}

\subsection{Survey on the Lane-Emden equation and connections}\label{integLE}

Named after the astrophysicists Lane \cite{lane} and Emden \cite{emden}, the Lane--Emden ODE
\neweq{lane-emden}
\frac{1}{\rho^2}\frac{d}{d\rho}\left[\rho^2\frac{d\theta}{d\rho}\right]+\theta^p=0\, ,\qquad\rho>0,\ p>1,
\endeq
describes in dimensionless form the Poisson equation for the gravitational potential of a Newtonian self-gravitating, spherically symmetric,
polytropic fluid; see also subsequent work by Fowler \cite{fowler,fowler2}.\par
The Lane-Emden equation \eq{lane-emden} is the radial version of the stationary superlinear PDE
\begin{equation}\label{anomalous}
-\Delta u=u^p,\quad u\ge0\mbox{ in }\R^3\, ,
\end{equation}
in which the exponent $p=5$ is critical since \eq{anomalous} is the Euler-Lagrange equation of the
``energy functional''
\neweq{energy}
J(u)=\frac{1}{2}\int_{\R^3}|\nabla u|^2-\frac{1}{p+1}\int_{\R^3}|u|^{p+1}
\endeq
defined over the space $\D$ which is the closure of $C^\infty_c(\R^3)$ with respect to the Dirichlet norm.
It is well-known that $\D\subset L^6(\R^3)$ with no embedding into other $L^p(\R^3)$-spaces; hence, the only
exponent for which \eq{energy} is well-defined is $p=5$ and its (nontrivial) critical points are the minimisers of the Sobolev ratio
$\|\nabla u\|_{L^2(\R^3)}^2/\|u\|_{L^6(\R^3)}^2$, thus leading to \eq{anomalous} for $p=5$.\par
When $p>1$ varies, different behaviours appear in \eq{anomalous}, as summarised in the next result.

\begin{proposition}\label{anomalointero}
$(i)$ If $1<p<5$, then the unique solution to \eqref{anomalous} is $u\equiv0$.\par\noindent
$(ii)$ If $p=5$, then for any $\xi_0\in\R^n$, \eqref{anomalous} admits infinitely many positive bounded solutions, radially symmetric and
decreasing with respect to $\xi_0\in\R^n$ explicitly given by
\neweq{bubble}
w_{\eps,\xi_0}(\xi)=\frac{\sqrt[4]{3\eps^2}}{\sqrt{\eps^2+|\xi-\xi_0|^2}}\qquad\forall\eps>0\, ,\ \forall\xi_0\in\R^3\, .
\endeq
Moreover, these are the only positive solutions to \eqref{anomalous} and
\neweq{Lq}
\begin{array}{ll}
w_{\eps,\xi_0}\in L^q(\R^3)\, \Longleftrightarrow\, q>3,\qquad & \|w_{\eps,\xi_0}\|_{L^6(\R^3)}\equiv C\quad\forall\eps>0,\\
\|w_{\eps,\xi_0}\|_{L^q(\R^3)}\stackrel{\eps\to0}{\to}0\quad\forall q\in(3,6),\qquad &
\|w_{\eps,\xi_0}\|_{L^q(\R^3)}\stackrel{\eps\to0}{\to}\infty\quad\forall q\in(6,\infty].
\end{array}
\endeq
$(iii)$ If $p>5$, then for any $\xi_0\in\R^n$, \eqref{anomalous} admits infinitely many positive bounded smooth solutions $u=u(\xi)$,
radially symmetric and decreasing with respect to $\xi_0\in\R^n$, all satisfying (as $|\xi|\to\infty$)
\neweq{asymptotic}
|\xi|^{\tfrac{2}{p-1}}u(|\xi|)\to\left[\tfrac{2(p-3)}{(p-1)^2}\right]^{\tfrac{1}{p-1}}\!,\
|\xi|^{\tfrac{p+1}{p-1}}\big|\nabla u(|\xi|)\big|\to\left[\tfrac{2^p(p-3)}{(p-1)^{p+1}}\right]^{\tfrac{1}{p-1}}\!,\
|\xi|^{\tfrac{2p}{p-1}}\big|\Delta u(|\xi|)\big|\to\left[\tfrac{2(p-3)}{(p-1)^2}\right]^{\tfrac{p}{p-1}}
\endeq
and, hence: $u\in L^q(\R^3)\Leftrightarrow q>\tfrac{3(p-1)}2$, $\nabla u\in L^q(\R^3)\Leftrightarrow q>\tfrac{3(p-1)}{p+1}$,
$\Delta u\in L^q(\R^3)\Leftrightarrow q>\tfrac{3(p-1)}{2p}$.
\end{proposition}
\begin{proof} For Item ($i$), see \cite[Theorem 1.1]{gs}.
Concerning Item ($ii$), for $p=5$ the solutions to \eq{anomalous} coincide with the so-called Talenti bubbles
\cite{talenti} given in \eq{bubble} (see also \cite{aubin} and their first appearance in \cite[(16)]{bliss}, earlier than the introduction of
Sobolev spaces in 1935!). By restricting to the case $\xi_0=0$, and simply denoting $w_\eps=w_{\eps,0}$, switching to spherical coordinates
we find
\neweq{similar}
\|w_\eps\|_{L^q(\R^3)}^q = 4\pi\, 3^{q/4}\, \eps^{q/2}\int_{0}^{\infty}\frac{\rho^2}{(\eps^2+\rho^2)^{q/2}}\, d\rho
\stackrel{\rho=\eps s}{=} 4\pi\, 3^{q/4}\, \eps^{3-q/2} \int_{0}^{\infty}\frac{s^2}{(1+s^2)^{q/2}}\, ds\, .
\endeq
Since we also have $\|w_{\eps}\|_{L^\infty(\R^3)}=w_\eps(0)=\sqrt[4]{3}/\sqrt{\eps}$, we infer \eqref{Lq}.
For Item ($iii$), we refer to \cite[Theorem 3.2]{ns2} and to the remark included in the proof of it, see also \cite{berpll} and
\cite[(E$_2$), p.551]{wang}. The proofs of \eq{asymptotic}$_2$-\eq{asymptotic}$_3$ can be derived from \eq{asymptotic}$_1$ after writing
\eq{anomalous} in its radial form \eq{lane-emden}.\end{proof}

Proposition \ref{anomalointero} shows that the critical exponent $p=5$ discriminates between
existence and nonexistence of positive solutions to \eqref{anomalous} (and also to \eq{lane-emden}). Due to \eq{anomalous}, the exponent at
the denominator of $u^5$ is the same as for the Laplacian, to be compared with $U$ in \eq{explicitexp} and $-\Delta U$.
The $L^3(\R^3)$-threshold of integrability for the bubble \eq{bubble}, see \eq{Lq}, is found also in the blow-up condition for the
unforced Navier-Stokes equations, see \cite{escauriaza,seregin}.
Incidentally we recall that, when $p=5$, there exist infinitely many sign-changing solutions to \eqref{anomalous}, see \cite{ding}.
As far as we are aware,
nothing is known about the possible existence of nonradial positive solutions to \eqref{anomalous} when $p>5$. Finally, by putting $p=5$ in
\eq{asymptotic} (which {\em is not} allowed!), we would obtain $u(|\xi|)\asymp|\xi|^{-1/2}$ as $|\xi|\to\infty$, while the decay of
\eq{bubble} is $w_{\eps,\xi_0}(|\xi|)\asymp|\xi|^{-1}$; therefore, there is discontinuity of the behaviour at infinity as $p\downarrow5$.

\begin{remark}\label{proportional}
One may wonder whether the bubble \eqref{bubble} can be replaced by a similar function with more integrability.
More precisely, consider the function $w(\xi)=(1+|\xi|^\beta)^{-\alpha}$ for some $\alpha,\beta>0$ since, otherwise, $w$ does not
vanish as $|\xi|\to\infty$. Then, by using radial coordinates, we find
$$
-\Delta w=\frac{\alpha\beta|\xi|^{\beta-2}[\beta+1+(1-\alpha\beta)|\xi|^\beta]}{(1+|\xi|^\beta)^{\alpha+2}}
$$
which is proportional to some power $p>1$ of $w$ (thereby satisfying \eqref{anomalous}) if and only if
$(\alpha,\beta)=(1/2,2)$, bringing us back to the bubble \eqref{bubble}. This is also connected with the proof of Lemma \ref{defforce}.
\end{remark}

Let us now recall some results about the superlinear parabolic Lane-Emden equation
\neweq{parabLE}
u_t-\Delta u=u^p\quad(p>1,\ u\ge0)\quad\mbox{in }\R^n\times\R_+\ ,\qquad u(\xi,0)=u_0(\xi)\quad\mbox{in }\R^n\, ,
\endeq
where $u^p$ should be replaced by $|u|^{p-1}u$ when seeking sign-changing solutions. Let us then explain why dissipation is not enough
to exclude blow up for \eq{parabLE}.
Although the energy $t\mapsto J\big(u(t)\big)$ ($J$ as in \eq{energy}) is
decreasing for smooth solutions to \eq{parabLE}, Wang \cite[Theorems 0.1-0.2-0.3]{wang} proved the following statement:

\begin{proposition}\label{alternativewang}
Assume that $p=3$. Then:\par\noindent
-- there exists $u_0\in C^0(\R^3)\cap L^\infty(\R^3)$ such that the unique local solution to \eqref{parabLE} is global, classical and
$\|u(t)\|_{L^\infty(\R^3)}\to0$ as $t\to\infty$;\par\noindent
-- there exists $u_0\in C^0(\R^3)\cap L^\infty(\R^3)$ and $T>0$ such that the unique local solution to \eqref{parabLE} is classical in
$\R^3\times(0,T)$ and $\|u(t)\|_{L^\infty(\R^3)}\to\infty$ as $t\to T$.
\end{proposition}

The blow-up for \eq{parabLE} is then intended as in \cite[p.772]{caffa}, see Theorem \ref{mainLinfty} and Remark \ref{sharpcaffa}. Although
the assumptions on $u_0$ involve pointwise inequalities, global existence or blow-up for \eq{parabLE} {\em are not} due to the comparison principle
which holds thanks to the positivity of the heat kernel since similar phenomena also appear in higher order parabolic equations
\cite{galaktionovpohozaev,gazgru}. In spite of the fact that polyharmonic heat kernels may
change sign \cite{fergazgru,gazgru3,gazgru4}, similar results for biharmonic Lane-Emden parabolic equations continue to hold
\cite{CM,egkp2,galaktionovpohozaev,gazgru2}. Moreover, (part of) the results in \cite{wang} have been extended
to {\em systems} of parabolic Lane-Emden equations \cite{andreucci,wangwang}. Hence, the positivity preserving property for the heat equation
is not strictly necessary and one may hope to find similar results also for vector fields and different nonlinearities; in
\cite[p.3244]{farwig} this difficulty was emphasised and attributed to the nonlocality of the Stokes operator \cite[p.3245]{farwig}.\par
Summarising, the behaviour of the solutions to these parabolic equations strongly depends on the nonlinearity and on the initial
datum but, quite paradoxically, it depends much less on the differential operator. This is how we built the blow-up examples for
\eqref{NS-intro}: the main idea was to employ modified versions of \eq{bubble}
(see also \cite[(3.12)]{leray} by choosing adequate exponents $p$ as for \eq{parabLE} and with a time-dependent
concentration parameter $\eps=\eps(t)$, see \eq{Lq}. We focused on positive powers of $|T-t|$, possibly different in the two appearances of
$\eps(t)$ but, perhaps, other choices could lead to better results.

\subsection{Proofs of Theorems \ref{gigasohr0} and \ref{gigasohr1}}\label{proofsgigasohr}

For $T>0$ and $(r,q)\in[1,\infty)\times(1,\infty)$ we define the Banach space
$$
\W^{r,q}_T:=\big\{U\in L_{\rm 1oc}^1(\R^3\times (0,T)):\, U\in W^{1,r}(0,T;L^q_\sigma(\R^3))\cap L^r(0,T;W^{2,q}(\R^3))\big\}
$$
with corresponding norm
$$
\left\|U\right\|_{\W^{r,q}_T}:=\left[\int_0^T\left(\|U_t(t)\|_{L^q(\R^3)}^r+\|U(t)\|_{W^{2,q}(\R^3)}^r\right)dt\right]^{1/r}\,.
$$

The following properties hold.

\begin{lemma}
Let $T>0$; there exists $c>0$ such that
$$
\|U\|_{L^{p}(0,T;L^s(\R^3))}+\|\nabla U\|_{L^{p_1}(0,T;L^{s_1}(\R^3))}\le c\|U\|_{\W^{r,q}_T}\qquad\forall U\in \W^{r,q}_T\, ,
$$
with
\neweq{E2}
\frac2{p}+\frac3{s}-\frac2{r}-\frac3{q}+2\ge 0\qquad\mbox{and}\qquad
\frac2{p_1}+\frac3{s_1}-\frac2{r}-\frac3{q}+1\ge 0\,,
\endeq
and where the strict inequality signs hold if either $p$ (resp. $p_1$) or $s$ (resp. $s_1$) is equal to $\infty$.\label{l:1}
\end{lemma}
\begin{proof} We show only the inequality for $U$, since the the proof
for that of $\nabla U$ is entirely analogous. To this end, we use an argument similar to that given in \cite[Theorem 2.1]{Solo}. Pick
$U\in\W^{r,q}$ and extend it to the entire space $(\xi,t)\in\R^4$, with preservation of (equivalent) norms, to get the following representation
\neweq{p_7}
U(\xi,t)=\int_{t-1}^t\int_{\R^3}\Gamma_1(\xi-\zeta,t-\tau)(U_\tau-\Delta U)(\zeta,\tau)d\zeta\, d\tau
+\int_{t-1}^t\int_{\R^3}\Gamma_2(\xi-\zeta,t-\tau)U(\zeta,\tau)d\zeta\, d\tau\,.
\endeq
Here,  $\psi:\R\to\R$ a smooth non-negative function such that $\psi\le 1$, $\psi(w)=0$ if $w\ge 1$ and $\psi(w)=1$ if $w\le 1/2$,
$\Gamma_1(\xi,t)=\Gamma (\xi,t)\psi(|\xi|)\psi(t)$,
$$
\Gamma_2(\xi,t)=2\psi(t)\nabla\Gamma(\xi,t)\cdot\nabla\psi(|\xi|)+\Gamma(\xi,t)\left(\psi(t)\Delta\psi(|\xi|)-\psi'(t)\psi(|\xi|)\right)\,,
$$
and
\neweq{p_8}
\Gamma(\xi,t)=\left\{\begin{array}{ll}
\frac{1}{4\pi t}\exp\left(-\frac{|\xi|^2}{4t}\right)&\ \ \mbox{if $t>0$}\\
0 &\ \ \mbox{if $t\le0$}
\end{array}\right.\,.
\endeq
By using  in \eqref{p_7} the generalised Minkowski inequality and Young's inequality for convolutions we get
$$
\|U(t)\|_{L^s(\R^3)}\!\le\!\int_{t-1}^t\!\!\!\big(\|\Gamma_1(t-\tau)\|_{L^\sigma(\R^3)}\!+\|\Gamma_2(t-\tau)\|_{L^\sigma(\R^3)}\big)
\big(\|U_\tau(\tau)\|_{L^q(\R^3)}\!+\|U(\tau)\|_{W^{2,q}(\R^3)}\big)d\tau,
$$
where $s\ge q$, and $1/\sigma=1+1/s-1/q$. As a result, taking into account that, by \eqref{p_8} and the properties of the function $\psi$,
$$
|\Gamma_i(\xi,\rho)|\le \frac{c_1}{(|\xi|^2+\rho)^{\frac32}}\,,\ \ \ i=1,2\,,
$$
we find
\neweq{p_9}
\|U(t)\|_{L^s(\R^3)}\le c_2\int_{t-1}^t\frac{\|U_\tau(\tau)\|_{L^q(\R^3)}+\|U(\tau)\|_{W^{2,q}(\R^3)}}{(t-\tau)^{\frac32(\frac1q-\frac1{s})}}d\tau\,,
\endeq
with $c_2$ independent of $t$. The stated bound \eq{E2} then follows by employing in \eqref{p_9} and the Hardy-Littlewood inequality if
\eqref{E2}$_1$ holds with the equality sign, and the Young inequality otherwise.\end{proof}

Take $\lambda>0$, $A:=\lambda-\Delta$: then, $D(A)=L^q_\sigma(\R^3)\cap W^{2,q}(\R^3)$ for any $1<q<\infty$ and
\begin{equation}
\|u\|_{W^{2,q}(\R^3)}\le C_1\,\|A\|_q\le C_2\,\|u\|_{W^{2,q}(\R^3)}\,,\ \ C_i=C_i(\lambda)\,,\ i=1,2.
\label{A}
\end{equation}
For $\gamma>1$ define
$$
D^{1-\frac1\gamma,\gamma}_q\!:=\Big\{u\in L^q_\sigma(\R^3):\|u\|_{D^{1-\frac1\gamma,\gamma}_q}:=
\|u\|_{L^q(\R^3)}+\Big(\int_0^\infty\|A\, {\rm e}^{-tA}u\|_{L^q(\R^3)}^\gamma\, dt\Big)^\frac1\gamma<\infty\Big\}\,.
$$

The following properties hold.

\begin{lemma}\label{l:2}
There exists $c=c(q,\gamma)>0$ such that
$$
\|u\|_{D^{1-\frac1\gamma,\gamma}_q}\le c\,\|u\|_{W^{2,q}(\R^3)}
\qquad\mbox{for all }u\in D(A)\,.
$$
\end{lemma}
\begin{proof} Since $A$ and ${\rm e}^{-tA}$ commute,  we have
\neweq{1}
\|A\, {\rm e}^{-tA}u\|_{L^q(\R^3)}=\| {\rm e}^{-tA}A\,u\|_{L^q(\R^3)}\le \|Au\|_{L^q(\R^3)}\le c(q,\lambda)\,\|u\|_{W^{2,q}(\R^3)}
\qquad\forall t\ge 0\,.
\endeq
Moreover, from classical results, we also have
\neweq{2}
\|A\, {\rm e}^{-tA}u\|_{L^q(\R^3)}\le c(q)\,t^{-1}\|u\|_{L^q(\R^3)}\le c(q)\, \|u\|_{L^q(\R^3)}\qquad\forall t>1\,.
\endeq
Thus, writing
$$
\int_0^\infty\|A\, {\rm e}^{-tA}u\|_{L^q(\R^3)}^\gamma dt=\int_0^1\|A\, {\rm e}^{-tA}u\|_{L^q(\R^3)}^\gamma dt
+\int_1^\infty\|A\, {\rm e}^{-tA}u\|_{L^q(\R^3)}^\gamma dt=:I_1+I_2\,,
$$
we employ \eqref{1} in $I_1$ and \eqref{2} in $I_2$ so that, taking into account that $\gamma>1$, the lemma follows.
\end{proof}

\begin{proof}[Proof of Theorem \ref{gigasohr0}] We seek a solution of the type $(V,P)=(u_S+v,p_S+\phi)\,{\rm e}^{\lambda\,t}$, $\lambda>0$, where $(u_S,p_S)$ solve the Stokes problem
\neweq{StP}
(u_S)_t+A u_S+\nabla p_S=f\,{\rm e}^{-\lambda\,t}\, ,\quad \nabla\cdot  u_S=0\quad\mbox{in }\R^3\times(0,T)\, ,\qquad
u_S(\xi,0)=V_0(\xi)\quad\mbox{for }\xi\in \R^3\, ,
\endeq
whereas $(v,\phi)$ satisfy
$$
\begin{array}{c}
v_t+A v+{\rm e}^{\lambda\,t}\big[(v\cdot\nabla)v+(u_S\cdot\nabla)v+(v\cdot\nabla)u_S+(u_S\cdot\nabla)u_S\big]+\nabla \phi=0\, ,\quad
\nabla\cdot  v=0 \quad\mbox{in }\R^3\times(0,T)\, ,\\
\qquad v(\xi,0)=0 \quad\mbox{for }\xi\in \R^3\, .
\end{array}
$$
From \cite[Theorem 2.7]{giga}, Lemma \ref{l:2} and (\ref{A}), we know that, under the given assumptions, \eq{StP} admits a unique solution
$u_S\in \mathcal W^{r,q}_T$. Furthermore, by Lemma \ref{l:1} and the assumption on the exponents $r,q$, we have
\neweq{3}
\|u_S\|_{L^\infty(\R^3\times(0,T))}\le c\,\|u_S\|_{\W^{r,q}_T}\le c\,(\|f\|_{\W^{r,q}_T}+\|V_0\|_{W^{2,q}(\R^3)})\,.
\endeq
A solution to \eqref{NS-intro} is then sought as a fixed point of the map
$$
M: w\in\W^{r,q}_T\to v\in\W^{r,q}_T
$$
with $v$ solving
\neweq{4}
\begin{array}{cc}
v_t+A v+{\rm e}^{\lambda\,t}[(w\cdot\nabla)w+(u_S\cdot\nabla)w+(w\cdot\nabla) u_S+(u_S\cdot\nabla)u_S]+\nabla \phi=0,\quad
\nabla\cdot  v=0\quad\mbox{in }\R^3\times(0,T),\\
v(\xi,0)=0\quad\mbox{for }\xi\in \R^3\, .
\end{array}
\endeq
From Lemma \ref{l:1} and \eqref{3}
we get
$$\begin{array}{cc}
\|(w\cdot\nabla)w+(u_S\cdot\nabla)w+(w\cdot\nabla)u_S+(u_S\cdot\nabla)u_S\|_{\W_T^{r,q}}\\
\le (\|w\|_{L^\infty(\R^3\times(0,T))}+\|u_S\|_{L^\infty(\R^3\times(0,T))})(\|w\|_{\W_T^{r,q}}+\|u_S\|_{\W_T^{r,q}})
\le c\,(\|w\|_{\W_T^{r,q}}^2+\|u_S\|_{\W_T^{r,q}}^2)\,.
\end{array}
$$
Therefore, applying \cite[Theorem 2.8]{giga} to \eqref{4} and taking into account (\ref{A}), we deduce
\neweq{ST}
\|v\|_{\W_T^{r,q}}\le c_1\,{\rm e}^{\lambda\,T}(\|w\|_{\W_T^{r,q}}^2+\|u_S\|_{\W_T^{r,q}}^2)\,.
\endeq
Next, if $\|w\|_{\W_T^{r,q}}\le R$ for some $R>0$, the previous inequality furnishes
$$
\|v\|_{\W_T^{r,q}}\le c_1\,{\rm e}^{\lambda\,T}(R^2+\|u_S\|_{\W_T^{r,q}}^2)\,,
$$
from which it follows that, if we take $R<{\rm e}^{-\lambda T}/2c_1$ and $T$ small enough that $\|u_S\|_{\W_T^{r,q}}<{\rm e}^{-\frac\lambda2 T}\sqrt{R/2c_1}$, then $M$ maps the ball of
radius $R$ centered at the origin of $\mathcal W^{r,q}$ into itself. Similarly, if we set
$$
v_i=M(w_i)\,,\ \, i=1,2\,,\ \,v=v_1-v_2\,,\ \ w=w_1-w_2
$$
and we use \eqref{4}, we get
$$
\|v\|_{\W_T^{r,q}}\le c_1{\rm e}^{\lambda T}\left(\|w_1\|_{\W_T^{r,q}}+\|w_2\|_{\W_T^{r,q}}+\|u_S\|_{\W_T^{r,q}}\right)\|w\|_{\W_T^{r,q}}
$$
which, again by taking $R$ and $T$ sufficiently small (say, $T\le T_*$), proves that $M$ is a contraction. Existence and uniqueness
are proved. Finally, \eqref{EST} is a consequence of \eqref{3}, \eqref{ST} and Lemma \ref{l:1}.\end{proof}

\begin{proof}[Proof of Theorem \ref{gigasohr1}] It is entirely analogous to that of Theorem \ref{gigasohr0}, once we observe the following facts.
(i) In view of the assumptions on the exponents $r,q$, from Lemma \ref{l:1} it follows
$$
\|\nabla u_S\|_{L^\infty(0,T\times\R^3)}\le c\,\|u_S\|_{\W^{r,q}_T}\,.
$$
(ii) Using again Lemma \ref{l:2} and the previous bound, we get
$$\begin{array}{c}
\|(w\cdot\nabla)w+(u_S\cdot\nabla)w+(w\cdot\nabla)u_S+(u_S\cdot\nabla)u_S\|_{\W_T^{r,q}}\\
\le \big[\|\nabla w\|_{L^\infty(\R^3\times(0,T))}+\|\nabla u_S\|_{L^\infty(\R^3\times(0,T))}\big]
\left[\|w\|_{\W_T^{r,q}}+\|u_S\|_{\W_T^{r,q}}\right]\le c\,(\|w\|_{\W_T^{r,q}}^2+\|u_S\|_{\W_T^{r,q}}^2)\,.
\end{array}
$$
The remaining details, including \eqref{ESTgrad}, are left to the reader.\end{proof}
\bigskip\noindent
{\bf Data availability statement.} Data sharing not applicable to this article as no datasets were generated or analysed during the current study.
\par\noindent
{\bf Conflict of interest statement}. The Authors declare that they have no conflict of interest.
\par\noindent
{\bf Acknowledgements.} G.P.G. is partially supported by NSF Grant DMS-2307811;  F.G.\ is supported by the MUR grant {\em Dipartimento di Eccellenza 2023-27} (Italy) and by INdAM.


{\small
\begin{thebibliography}{999}
\bibitem{albritton} D. Albritton, E. Bru\'{e}, M. Colombo, {\em Non-uniqueness of {L}eray solutions of the forced
{N}avier-{S}tokes equations}, Ann. of Math. 196, 415-455, 2022
\bibitem{andreucci} D. Andreucci, M.A. Herrero, J.J.L. Velázquez, {\em Liouville theorems and blow up behaviour in semilinear reaction
diffusion systems}, Ann. Inst. H. Poincaré Anal. Non Linéaire 14, 1–53, 1997
\bibitem{aubin} T. Aubin, {\em Problèmes isopérimétriques et espaces de Sobolev}, J. Differ. Geom. 11, 573–598, 1976
\bibitem{majda} J.T. Beale, T. Kato, A. Majda, {\em Remarks on the breakdown of smooth solutions for the 3-{D} {E}uler equations},
Comm. Math. Phys. 94, 61-66, 1984
\bibitem{Hugo} H. Beir{\~a}o da Veiga, J. Yang, {\em A note on the development of singularities on solutions to the Navier-Stokes equations
under super critical forcing terms}, arXiv:2411.10823v1 [math.AP]
\bibitem{bliss} G.A. Bliss, {\em An integral inequality}, J. London Math. Soc. 5, 40-46, 1930
\bibitem{berpll} H. Berestycki, P.L. Lions, L.A. Peletier, {\em An {ODE} approach to the existence of positive solutions for
semilinear problems in {${\bf R}^{N}$}}, Indiana Univ. Math. J. 30, 141-157, 1981
\bibitem{caffa} L. Caffarelli, R. Kohn, L. Nirenberg, {\em Partial regularity of suitable weak solutions of the {N}avier-{S}tokes equations},
Comm. Pure Appl. Math. 35, 771-831, 1982
\bibitem{CM} G. Caristi, E. Mitidieri, {\em Existence and nonexistence of global solutions of higher-order parabolic
problems with slow decay initial data}, J. Math. Anal. Appl. 279, 710-722, 2003
\bibitem{coiculescu} M.P. Coiculescu, S. Palasek, {\em Non-uniqueness of smooth solutions of the {N}avier-{S}tokes
equations from critical data}, Invent. Math. 244, 165-219, 2026
\bibitem{cordoba} D. C\'ordoba, L. Mart\'inez-Zoroa, {\em Blow-up for the incompressible 3D-Euler equations with uniform
$C^{1,\frac{1}{2}-\epsilon}\cap L^2$ force}, arXiv:2309.08495v1 [math.AP]
\bibitem{ding} W.Y. Ding, {\em On a conformally invariant elliptic equation on $\R^n$}, Comm. Math. Phys. 107, 331-335, 1986
\bibitem{egkp2} Y.V. Egorov, V.A. Galaktionov, V.A. Kondratiev, S.I. Poho\v zaev, {\em Global solutions of higher-order
semilinear parabolic equations in the supercritical range}, Adv. Diff. Eq. 9, 1009-1038, 2004
\bibitem{emden} R. Emden, {\em Gaskugeln: Anwendungen der Mechanischen W\"armetheorie auf Kosmologische und Meteorologische Probleme},
Teubner, Berlin, 1907. Reprinted by Forgotten Books, 2018
\bibitem{escauriaza} L. Escauriaza, G.A. Ser\"{e}gin, V. Shverak, {\em $L_{3,\infty}$-solutions of {N}avier-{S}tokes equations and
backward uniqueness}, Uspekhi Mat. Nauk 58, 3-44, 2003
\bibitem{farwig} R. Farwig, {\em From {J}ean {L}eray to the millennium problem: the {N}avier-{S}tokes equations},
J. Evol. Equ. 21, 3243-3263, 2021
\bibitem{feffer} C.L. Fefferman, {\em Existence and smoothness of the Navier-Stokes equation}, The millennium prize problems,
Clay Math. Inst., Cambridge, MA, 57-67, 2006
\bibitem{fergazgru} A. Ferrero, F. Gazzola, H.C. Grunau, {\em Decay and eventual local positivity for biharmonic parabolic
equations}, Disc. Cont. Dynam. Syst. A 21, 1129-1157, 2008
\bibitem{fowler} R. Fowler, {\em The solution of Emden's and similar differential equations}, Monthly Notices Roy. Astronom. Soc. 91, 63-91, 1930
\bibitem{fowler2} R. Fowler, {\em Further studies of Emden's and similar differential equations}, Quarterly J. Math. Oxford Ser. 2, 259-288, 1931
\bibitem{fujiwara} D. Fujiwara, H. Morimoto, {\em An {$L_{r}$}-theorem of the {H}elmholtz decomposition of vector fields},
J. Fac. Sci. Univ. Tokyo Sect. IA Math. 24, 685-700, 1977
\bibitem{fursikov} A.V. Fursikov, {\em On some control problems and results concerning the unique solvability of a mixed boundary value
problem for the three-dimensional Navier–Stokes and Euler systems}, Dokl. Akad. Nauk SSSR 252:5, 1066–1070, 1980
\bibitem{fursikov2} A.V. Fursikov, {\em Control problems and theorems concerning the unique solvability of a mixed boundary value
problem for the three-dimensional Navier–Stokes and Euler equations}, Math. USSR Sbornik 43, 251–273, 1982
\bibitem{galaktionovpohozaev} V.A. Galaktionov, S.I. Poho\v{z}aev, {\em Existence and blow-up for higher-order semilinear parabolic equations:
majorizing order-preserving operators}, Indiana Univ. Math. J. 51, 1321-1338, 2002
\bibitem{Galdi-evol} G.P. Galdi, {\em An Introduction to the Navier-Stokes Initial-Boundary Value Problem}, In:
G.P. Galdi, J.G. Heywood, R. Rannacher (eds), Fundamental Directions in Mathematical Fluid Mechanics.
Advances in Mathematical Fluid Mechanics. Birkh\"auser, Basel, 1-70, 2000
\bibitem{Galdimeta} G.P. Galdi, G. Metafune, C. Spina, C. Tacelli, {\em Homogeneous {C}alder\'{o}n-{Z}ygmund estimates for a class of
second-order elliptic operators}, Commun. Contemp. Math. 17, 1450017, 14pp., 2015
\bibitem{gazgru} F. Gazzola, H.C. Grunau, {\em Radial entire solutions for supercritical biharmonic equations}, Math. Ann. 334, 905-936, 2006
\bibitem{gazgru2} F. Gazzola, H.C. Grunau, {\em Global solutions for superlinear parabolic equations involving the biharmonic
operator for initial data with optimal slow decay}, Calc. Var. Part. Diff. Eq. 30, 389-415, 2007
\bibitem{gazgru3} F. Gazzola, H.C. Grunau, {\em Eventual local positivity for a biharmonic heat equation in $\R^n$},
Disc. Cont. Dynam. Syst. S 1, 83-87, 2008
\bibitem{gazgru4} F. Gazzola, H.C. Grunau, {\em Some new properties of biharmonic heat kernels}, Nonlin. Anal. TMA 70, 2965-2973, 2009
\bibitem{gs} B. Gidas, J. Spruck, {\em Global and local behavior of positive solutions of nonlinear elliptic equations},
Comm. Pure Appl. Math. 34, 525-598, 1981
\bibitem{giga} Y. Giga, H. Sohr, {\em Abstract {$L^p$} estimates for the {C}auchy problem with applications to the {N}avier-{S}tokes equations in exterior domains},
J. Funct. Anal. 102, 72-94, 1991
\bibitem{hopf} E. Hopf, {\em \"Uber die anfangswertaufgabe f\"ur die hydrodynamischen grundgleichungen}, Math. Nachr. 4, 213-231 (1951)
\bibitem{hou} T. Hou, Y. Wang, C. Yang, {\em Nonuniqueness of Leray-Hopf solutions to the unforced incompressible 3D Navier-Stokes equation},
arXiv:2509.25116v2 [math.AP]
\bibitem{kozono} H. Kozono, H. Sohr, {\em Remark on uniqueness of weak solutions to the {N}avier-{S}tokes equations}, Analysis 16, 3,
255-271, 1996
\bibitem{lane} J.H. Lane, {\em On the theoretical temperature of the Sun, under the hypothesis of a gaseous mass maintaining its
volume by its internal heat, and depending on the laws of gases as known to terrestrial experiment}, American Journal of Science 50 (148),
57–74, 1870
\bibitem{leray} J. Leray, {\em Sur le mouvement d'un liquide visqueux emplissant l'espace}, Acta Mathematica 63, 193-248, 1934
\bibitem{lions} J.L. Lions, {\em Sur la régularité et l'unicité des solutions turbulentes des équations de {N}avier {S}tokes},
Rend. Sem. Mat. Univ. Padova 30, 16-23, 1960
\bibitem{miyakawa} T. Miyakawa, {\em On nonstationary solutions of the {N}avier-{S}tokes equations in an exterior domain},
Hiroshima Math. J. 12, 115-140, 1982
\bibitem{ns2} W.M. Ni, J. Serrin, {\em Existence and nonexistence theorems for ground states for quasilinear partial differential equations.
The anomalous case}, Accad. Naz. Lincei, Atti dei Convegni 77, 231-257, 1986
\bibitem{seregin} G. Seregin, {\em A certain necessary condition of potential blow up for Navier-Stokes equations},
Comm. Math. Phys. 312, 833-845, 2012
\bibitem{serrin2} J. Serrin, {\em The initial value problem for the Navier-Stokes equations}, In: Nonlinear Problems (Proc. Sympos.,
Madison, 1962), Univ. Wisconsin Press, Madison, 69–98, 1963
\bibitem{simader} C. Simader, H. Sohr, {\em A new approach to the {H}elmholtz decomposition and the {N}eumann problem in {$L^q$}-spaces
for bounded and exterior domains}, In: Mathematical problems relating to the {N}avier-{S}tokes equation, Ser. Adv. Math. Appl. Sci.
11, 1-35, World Sci. Publ., River Edge, NJ, 1992
\bibitem{sohr} H. Sohr, {\em The {N}avier-{S}tokes equations - An elementary functional analytic approach}, Birkh\"{a}user Advanced Texts:
Basel Textbooks, 2001
\bibitem{Solo} V.A. Solonnikov, {\em Estimates of the solutions of the nonstationary Navier-Stokes system} (Russian)
{\em Boundary value problems of mathematical physics and related questions in the theory of functions},
Zap. Naucn. Sem. Leningrad. Otdel. Mat. Inst. Steklov (LOMI), 38, 153-231, 1973
\bibitem{talenti} G. Talenti, {\em Best constant in Sobolev inequality}, Ann. Mat. Pura Appl. 110, 353-372, 1976
\bibitem{wang} X. Wang, {\em On the {C}auchy problem for reaction-diffusion equations}, Trans. Amer. Math. Soc. 337, 549-590, 1993
\bibitem{wangwang} X. Wang, J. Wang, {\em Global existence and stability analysis for the parabolic Lane–Emden system with initial data},
Nonlinearity 39, 4, 045008, 2026
\bibitem{zhang} Q.S. Zhang, {\em A blow up solution of the {N}avier-{S}tokes equations with a super critical forcing term},
J. Math. Study 58, 4, 429-438, 2025
\end{thebibliography}
}
\end{document}